\documentclass{amsart}
\usepackage{amsfonts,amssymb,amsmath,amsthm}
\usepackage{url}
\usepackage{enumerate}

\newtheorem{thrm}{Theorem}[section]
\newtheorem{lem}[thrm]{Lemma}
\newtheorem{prop}[thrm]{Proposition}

\theoremstyle{definition}
\newtheorem{definition}[thrm]{Definition}
\newtheorem{remark}[thrm]{Remark}

\numberwithin{equation}{section}

\author{A. Ivanov}
\address{
Institute of Mathematics\\
Silesian University of Technology\\
ul. Kaszubska 23, Gliwice, 44-101 Poland}
\email{Aleksander.Iwanow@polsl.pl}
\thanks{The author is supported by Polish National Science Centre grant DEC2011/01/B/ST1/01406}


\begin{document}

\title[Operator expansions of finite dimensional Hilbert spaces]{Continuous theory of operator expansions of finite dimensional Hilbert spaces, 
continuous structures of quantum circuits and decidability}

\begin{abstract}
We consider continuous structures which are obtained 
from finite dimensional Hilbert spaces over $\mathbb{C}$ 
by adding some unitary operators. 
Quantum automata and circuits are naturally interpretable 
in such structures. 
We consider appropriate algorithmic problems concerning 
continuous theories of natural classes of these structures. 
\end{abstract}
\maketitle

\section{Introduction} \label{sect1}

Continuous logic has become the basic model theoretic 
tool for Hilbert spaces and ${\bf C}^*$-algebras: 
see \cite{BYBHU}, \cite{BYU} and \cite{FHS}. 
This suggests that quantum circuits, quantum automata 
and quantum computations in general 
can be defined in appropriate continuous structures 
and studied by means of continuous logic.  
The paper presents an attempt of this approach. 
The main object of our paper are finite dimensional 
Hilbert spaces in the language expanded by 
a finite family of unitary operators. 
We call them dynamical Hilbert spaces. 
\parskip0pt 

It is worth noting that a finite dimensional  
 Hilbert space cannot be considered as 
an object interesting on its own from the point of view 
of continuous model theory.  
This case corresponds to 'finite objects' in model theory  
(its $n$-balls are compact). 
In our paper we study continuous theories of 
{\em classes} of these structures. 
This naturally leads to pseudo finite dimensional 
structures and to questions connected with 
approximations of groups by metric groups. 
\parskip0pt 

All necessary information on  
continuous logic  will be described in 
the next section. 

The main results of the paper 
concern decidability of continuous theories of 
classes of dynamical Hilbert spaces and 
so called 'marked dynamical Hilbert spaces'. 
In Section 4 we show that decidability questions for the class of 
finite dimensional dynamical Hilbert spaces are connected with 
property MF, one of the most interesting properties in 
the topic of approximations by metric groups \cite{CDE}.  
Marked dynamical Hilbert  spaces are defined in Section 5 
as expansions of finite dimensional dynamical Hilbert spaces 
by unary discrete predicates. 
We will see that this procedure is essential 
for expressive power of the language.  
In particular there are natural subclasses 
of marked dynamical Hilbert spaces where undecidable 
first order theories of some classes of finite structures 
can be interpreted.  
These results are partially motivated by \cite{DJK}, 
where algorithmic problems for quantum automata were studied. 
In the beginning of Section 5 we give 
a more detailed introduction to 
these issues. 

Section 3 contains some general 
observations concerning decidability.  
We think that this section is interesting by itself. 
It is naturally connected with the material 
of \cite{BYP}, \cite{DGP} and \cite{GH}, 
where decidability questions for continuous theories were initiated.  

The author is grateful to Isaac Goldbring for the suggestion that 
the universal theory of dynamical (finite dimensional) Hilbert spaces is decidable and to Udi Hrushovski for several remarks concerning pseudocompactness (see Section 4.4).

\section{Continuous structures.} 

\subsection{General preliminaries} 
We fix a countable continuous signature 
$$
L=\{ d,R_1 ,...,R_k ,..., F_1 ,..., F_l ,...\}. 
$$  
Let us recall that a {\em metric $L$-structure} 
is a complete metric space $(M,d)$ with $d$ bounded by 1, 
along with a family of uniformly continuous operations on $M$ 
and a family of predicates $R_i$, i.e. uniformly continuous maps 
from appropriate $M^{k_i}$ to $[0,1]$.   
It is usually assumed that $L$ assigns to each predicate symbol $R_i$ 
a continuity modulus $\gamma_i :[0,1] \rightarrow [0,1]$ so that 
any metric structure $M$ of the signature $L$ satisfies the property  
that if $d(x_j ,x'_j ) <\gamma_i (\varepsilon )$ with $1\le j\le k_i$ 
then the inequality 
$$ 
|R_i (x_1 ,...,x_j ,...,x_{k_i}) - R_i (x_1 ,...,x'_j ,...,x_{k_i})| < \varepsilon  
$$ 
holds for the corresponding predicate of $M$.  
It happens very often that $\gamma_i$ coincides with $id$. 
In this case we do not mention the appropriate modulus. 
Similarly, the language also includes continuity moduli for functional symbols. 

Note that each countable structure can be considered 
as a complete metric structure with the discrete $\{ 0,1\}$-metric. 

By completeness continuous substructures of a continuous structure are always closed subsets. 

Atomic formulas are the expressions of the form $R_i (t_1 ,...,t_r )$, 
$d(t_1 ,t_2 )$, where $t_i$ are simply classical terms (built from functional $L$-symbols). 
We define {\em formulas} to be expressions built from 
0,1 and atomic formulas by applications of the following functions: 
$$ 
x/2  \mbox{ , } x\dot- y= \mathsf{max} (x-y,0) \mbox{ , } \mathsf{min}(x ,y )  \mbox{ , } \mathsf{max}(x ,y )
\mbox{ , } |x-y| \mbox{ , } 
$$ 
$$ 
\neg (x) =1-x \mbox{ , } x\dot+ y= \mathsf{min}(x+y, 1) \mbox{ , } x \cdot y \mbox{ , } \mathsf{sup}_x \mbox{ and } \mathsf{inf}_x . 
$$   
{\em Statements} concerning metric structures are usually 
formulated in the form 
$$
\phi = 0 , 
$$ 
where $\phi$ is a formula. 
Sometimes statements are called an {\em condition} or $\bar{x}$-conditions 
(when $\phi$ depends on $\bar{x}$); we will use both names. 
A {\em theory} is a set of statements without free variables 
(here $\mathsf{sup}_x$ and $\mathsf{inf}_x$ play the role of quantifiers). 
If $\mathcal{K}$ is a class of continuous $L$-structures then 
$Th(\mathcal{ K})$ denotes the set of all conditions without free variables 
which hold in all structures of $\mathcal{ K}$. 
   
We sometimes replace conditions of the form $\phi \dot{-} \varepsilon =0$ 
where $\varepsilon \in [0,1]$ by more convenient 
expressions $\phi \le \varepsilon$.    
When a formula $\phi$ is of the form 
$\mathsf{sup}_{x_1}  \mathsf{sup}_{x_2}\ldots \mathsf{sup}_{x_1} \psi$, 
where $\psi$ is quantifier free, we say that $\phi$ is {\em universal}.
   
It is worth noting that any formula is a $\gamma$-uniformly continuous 
function from the appropriate power of $M$ to $[0,1]$, 
where $\gamma$ is the minimum of continuity moduli of $L$-symbols 
appearing in the formula. 

The condition that the metric is bounded by $1$ is not necessary. 
It is often assumed that $d$ is bounded by some rational number $d_0$. 
In this case the (truncated) functions above are appropriately modified.  
Sometimes predicates of continuous structures map $M^n$ to some 
$[q_1 ,q_2 ]$ where $q_1 ,q_2 \in \mathbb{Q}$.  
It is only worth noting that we always assume that when we fix 
an interval $[q_1 ,q_2 ]$ for values of continuous formulas, 
connectives are chosen so that they  cannot give values outside this interval. 

\bigskip 

Following Section 4.2 of \cite{FHS} we define a topology 
on $L$-formulas relative to a given continuous theory $T$. 
For $n$-ary formulas $\phi$ and $\psi$ of the same sort set 
$$ 
{\bf d}^T_{\bar{x}} (\phi ,\psi) = \mathsf{sup} \{ |\phi (\bar{a}) -\psi (\bar{a} )|: \bar{a} \in M, M\models T\} . 
$$ 
The function ${\bf d}^T_{\bar{x}}$ is a pseudometric. 

\begin{definition} \label{FaHaSh} 
The language $L$ is called {\em separable} with respect to  
$T$ if for any tuple $\bar{x}$ 
the density character of ${\bf d}^T_{\bar{x}}$ is countable. 
\end{definition} 

By Proposition 4.5 of \cite{FHS} when $L$ is separable, 
for every $M\models T$ the set of all interpretations of $L$-formulas 
in $M$ is separable in the uniform topology. 

\bigskip 

The paper \cite{BYP} gives  
fourteen axioms of continuous 
first order logic, denoted by  
(A1) - (A14), and the corresponding version of 
{\em modus ponens}: 
$$ 
\frac{\phi \mbox{ , } \psi \dot{-} \phi}{\psi} . 
\mbox{ where } \phi , \psi \mbox{ are continuous formulas.} 
$$ 
Corollary 9.6 of \cite{BYP} states:  
\begin{quote} 
Let $\Gamma$ be a set of continuous formulas 
of a continuous signature $L$ with a metric. 
Let $\phi$ be a continuous $L$-formula. 
Then the following conditions are equivalent: \\ 
(i) for any continuous structure $M$ and any $M$-assignment of variables, 
if $M$ satisfies all statements $\psi = 0$, 
$\psi \in \Gamma$, then $M$ satisfies $\phi =0$; \\ 
(ii) $\Gamma \vdash \phi \dot{-} 2^{-n}$ for all $n\in \omega$. 
\end{quote} 
It is called {\em approximated strong 
completeness for continuous first-order logic}. 
The following statement is Corollary 9.8 from \cite{BYP}. 
\begin{quote} 
Under circumstances above 
the following values are the same: \\  
(i) $\mathsf{sup} \{ \phi^{M} :  \mbox{ for all } M \models \Gamma =0\}$; \\ 
(ii) $\mathsf{inf} \{ p\in \mathbb{Q}: \Gamma \vdash   \phi \dot{-} p\}$. 
\end{quote} 
We denote this value by $\phi^{\circ}$ and call it the {\em degree of truth  of} $\phi$ {\em with respect to} $\Gamma$.  

If the language $L$ is computable, the set of 
all continuous $L$-formulas and the set of all $L$-conditions of the form 
$$ 
\phi \le \frac{m}{n} \mbox{ , where } \frac{m}{n}\in \mathbb{Q}_+ , 
$$ 
are computable. 
Moreover if $\Gamma$ is a computably enumerable 
set of formulas, then the set $\{ \phi : \Gamma \vdash \phi \}$ 
is computably enumerable. 

Corollary 9.11 of \cite{BYP} states that when $\Gamma$ 
is computably enumerable 
and $\Gamma=0$ axiomatizes a complete theory, 
then the value of $\phi$ in models of $\Gamma =0$ is 
a recursive real which is  uniformly computable from $\phi$.  
This exactly means that the corresponding 
complete theory is {\em decidable} (see Section 2). 
Note that in this case the value of $\phi$ as above coincides with 
$\phi^{\circ}$.

%\bigskip 

\subsection{Hilbert spaces} 
We treat a Hilbert space over $\mathbb{R}$ 
exactly as in Section 15 of \cite{BYBHU}. 
We identify it with a many-sorted metric structure 
$$
(\{ B_n\}_{n\in \omega} ,0,\{ I_{mn} \}_{m<n} ,
\{ \lambda_r \}_{r\in\mathbb{R}}, +,-,\langle \rangle ),
$$
where $B_n$ is the ball of elements of norm $\le n$, 
$I_{mn}: B_m\rightarrow B_n$ is the inclusion map, 
$\lambda_{r}: B_m\rightarrow B_{km}$ is scalar 
multiplication by $r$, with $k$ the unique integer 
satisfying $k\ge 1$ and $k-1 \le |r|<k$; 
furthermore, $+,- : B_n \times B_n \rightarrow B_{2n}$ 
are vector addition and subtraction and 
$\langle \rangle : B_n \rightarrow [-n^2 ,n^2 ]$ 
is the predicate of the inner product. 
The metric on each sort is given by 
$d(x,y) =\sqrt{ \langle x-y, x-y \rangle }$.   
For every operation the continuity modulus is standard.  
For example in the case of $\lambda_r$ this is $\frac{z}{|r|}$. 

Stating existence of infinite approximations 
of orthonormal bases by axioms of the form 
$$ 
\mathsf{inf}_{x_1 ,...,x_n\in B_1} \mathsf{max}_{1\le i<j\le n} (|\langle x_i ,x_j\rangle -\delta_{i,j} |) =0 \mbox{ , } n\in \omega , 
$$
$$ 
\delta_{i,j} \in \{ 0,1\} \mbox{ with } \delta_{i,j} =1 \leftrightarrow i=j , 
$$   
we axiomatize  infinite dimensional Hilbert spaces. 
By \cite{BYBHU} they form the class of models of a complete 
theory which is $\kappa$-categorical for all infinite $\kappa$, 
and admits elimination of quantifiers. 

When we assume that the space is finite dimensional 
all sorts $B_n$ become compact. 
This corresponds to the case 
of finite structures in ordinary 
model theory. 
The statement that the dimension equals $n$ 
can be described by the following statement.  
$$ 
\mathsf{inf}_{y_1 ,...,y_n \in B_1} 
\mathsf{max} (\mathsf{max}_{1\le i\le n} (|\langle y_i ,y_i\rangle -1|), 
$$
$$
\mathsf{sup}_{x\in B_1} (| ( \langle x,x\rangle  - 
|\langle x,y_1 \rangle |^2 - ... ...-|\langle x, y_n \rangle |^2 ) |)=0 .
$$ 
The corresponding continuous theory 
admits elimination of quantifiers. 
This follows by the argument of  
Lemma 15.1 from \cite{BYBHU}.

This approach can be naturally extended to complex Hilbert spaces, 
$$
(\{ B_n\}_{n\in \omega} ,0,\{ I_{mn} \}_{m<n} ,
\{ \lambda_c \}_{c\in\mathbb{C}}, +,-,\langle \rangle_{Re} , \langle \rangle_{Im} ). 
$$
We only extend the family 
$\lambda_{r}: B_m\rightarrow B_{km}$, $r\in \mathbb{R}$, 
to a family $\lambda_{c}: B_m\rightarrow B_{km}$, $c\in \mathbb{C}$, 
of scalar products by $c\in\mathbb{C}$, with $k$ 
the unique integer satisfying $k\ge 1$ and $k-1 \le |c|<k$. 

We also introduce $Re$- and $Im$-parts of the inner product. 

If we remove from the signature of complex Hilbert spaces 
all scalar products by $c\in \mathbb{C}\setminus \mathbb{Q}[i]$, 
we obtain a countable subsignature 
$$
(\{ B_n\}_{n\in \omega} ,0,\{ I_{mn} \}_{m<n} ,
\{ \lambda_c \}_{c\in\mathbb{Q}[i]}, +,-,\langle \rangle_{Re} , \langle \rangle_{Im} ),
$$
which is {\em dense} in the original one: \\ 
if we present $c\in \mathbb{C}$ by a sequence $\{ q_i \}$ 
from $\mathbb{Q}[i]$ converging to $c$, 
then the choice of the continuity moduli of 
the restricted signature still guarantees that 
in any sort $B_n$ the functions $\lambda_{q_i}$ 
form a sequence which converges to  $\lambda_c$ 
with respect to the metric 
$$ 
\mathsf{sup}_{x\in B_n} \{ |f^M (x) - g^M (x)| : M \mbox{ is a model of the theory of Hilbert spaces } \}.  
$$  
This obviously implies that the original language 
of Hilbert spaces is {\em separable}.  

\bigskip 

To study dynamical evolutions of quantum circuits 
we introduce the following expansion of 
Hilbert spaces.
Let us fix a natural number $t$ and consider 
the class of {\em dynamical Hilbert spaces} 
%(i.e. $2^n$-dimensional) 
in the extended signature 
$$
( \{ B_n\}_{n\in \omega} ,0,\{ I_{mn} \}_{m<n} ,
\{ \lambda_c \}_{c\in\mathbb{Q}[i]}, +,-,\langle \rangle_{Re} , \langle \rangle_{Im}, U_1 ,...,U_t  ),
$$
where $U_j$, $1\le j \le t$,  are  symbols of unitary 
operators of $\mathbb{H}$.  
We may assume that all $U_j$ are defined only on $B_1$. 
For convenience we add to each $U_i$ the symbol 
$U'_{i}$ for the operator $U^{-1}_i$. 
Then we also add the axioms 
$\mathsf{sup}_{v\in B_1} d(U'_i U_i (v) ,v) \le 0$ and 
$\mathsf{sup}_{v\in B_1} d(U_i U'_i (v) ,v) \le 0$. 
We will not mention this below.  

It is clear that this language is computable and  
is dense in the $\bar{U}$-extension of the standard language of 
the theory of Hilbert spaces. 
The main results of the paper concern decidability of 
theories in this language.

\begin{lem} \label{comp} 
Assume that a structure of the form above 
is $n$-dimensional where $n\in \mathbb{N}$.  
Then the complete continuous theory of this structure 
is axiomatized by the standard axioms of Hilbert spaces, 
the axioms stating that each 
$U_j$ is a unitary operator and the following axioms 
describing the matrices of $U_j$ in some 
(fixed) orthogonal normal  basis: 
$$ 
\mathsf{inf}_{y_1 ,...,y_n \in B_1} \mathsf{max}
(\mathsf{max}_{1\le i\le n} (|\langle y_i ,y_i\rangle -1|), 
$$
$$
\mathsf{sup}_{x\in B_1} (| ( \langle x,x\rangle  - |\langle x,y_1 \rangle |^2 - ... 
...-|\langle x, y_n \rangle |^2 ) |), 
$$
$$ 
\mathsf{max}_{1\le l\le n} \mathsf{max}_{1\le j\le t} (\parallel U_j (y_l )-\sum \lambda_{c_{j,l,k}} (y_k ) \parallel \dot{-} \varepsilon_l ))\le 0 , 
$$ 
$$
\mbox{ where } \varepsilon_l \in \mathbb{Q} 
\mbox{ and } c_{j,l,k} \in \mathbb{Q}[i] 
\mbox{ are appropriate approximations } 
$$ 
$$ 
\mbox{ of entries of matrices for } 
U_1 ,\ldots ,U_t . 
$$ 
\end{lem} 

\begin{proof} 
Any model with these axioms is 
an $n$-dimensional space. 
Thus by compactness of $B_1$ there is an appropriate basis 
where the values of $U_j (y_l )$ have the coordinates  
described in the axioms. 
This model is unique up to isometry. 
Thus the lemma is obvious. 
\end{proof} 

\subsection{Unitary representations} 
When we consider a language containing countably many 
operators $U_i$, any unitary representation of a countable group $G$ 
can be considered as a dynamical Hilbert space in this language. 
For example we can add an operator for every element of $G$. 

We will use several notions from the area of unitary representations. 
We firstly remind the reader that the left regular representation of $G$ 
is obtained by the action of $G$ on $l^2 (G)$ 
defined by the unitary operators $U_g : f(h) \rightarrow f(g^{-1}h)$.  
The $*$-algebra generated by all $U_g$ 
is just $\mathbb{C}G$. 

The following notion is taken from Section F of \cite{BHV}. 

\begin{definition} \label{alco} 
Let $\pi$ and $\rho$ be unitary representations of $G$ and 
$H_{\rho}$, $H_{\pi}$ be the corresponding dynamical Hilbert spaces.  
Let $\varepsilon > 0$ and $F $ be a finite subset of $G$. 
We say that $\rho$ is $(\varepsilon , F)$-{\em contained in} $\pi$ 
if for every $v_1 , \ldots , v_n \in H_{\rho}$, there
are $w_1 ,\ldots ,w_n \in H_{\pi}$ such that 
$$ 
|\langle \rho (g)v_i , v_j\rangle - \langle \pi (g)w_i ,w_j\rangle | < \varepsilon 
\mbox{  for all }g \in F. 
$$ 
We say that $\rho$ is {\em weakly contained in} $\pi$ and write 
$\rho \prec \pi$ if $\rho$ is $(\varepsilon , F)$-contained 
in $\pi$ for every $\varepsilon$ and finite  $F$.
\end{definition} 

When $G$ is $t$-generated we apply the definition above 
to the corresponding homomorphisms of the form  
$G \rightarrow \langle U_1 \ldots , U_t \rangle$. 

In the case when $t=1$ some standard material from functional 
analysis can be applied. 
We remind the reader that a complex number $\lambda$ 
is said to be a {\em regular value} of a operator $U$ if
there exists $(U -\lambda Id )^{-1}$, which is a bounded linear 
operator and is defined on a dense subspace of the space. 
The {\em resolvent set} of $U$ is the set of all regular values of $U$. 
The spectrum of $U$, denoted by $\sigma (U)$, 
is the complement of the resolvent set. 
The set of isolated points of $\sigma (U)$ of finite multiplicity 
is called the {\em finite spectrum} and is denoted by $\sigma_{fin} (U)$. 
The set $\sigma_e (U) = \sigma (U) \setminus \sigma_{fin} (U)$ 
is called the {\em essential spectrum} of $U$.

We will use this material in combination with the following theorem of C. Ward Henson. 
\begin{quote} 
Let $(H ,U)$ and $(H',U')$ be dynamical Hilbert spaces with one operator. 
These spaces are elementarily equivalent in continuous logic if and only if 
they have the same essential spectra $\sigma_e (U)$ and $\sigma_e (U')$ and 
for any $r\in S^1 \setminus \sigma_e (U)$ we have 
$$
\mathsf{dim} \{ x\in H: Ux = rx \} = \mathsf{dim} \{ c\in H' : U'x = rx \}. 
$$  
\end{quote} 
A proof of this theorem in the case of countable spectrum can be found 
in \cite{argober}.

\section{Decidability/undecidability of continuous theories} 

In this section we assume that the signature $L$ 
is computable and values of formulas are in $[0,1]$. 
The interval $[0,1]$ can be obviously replaced 
by any compact interval. 
We start with the following definition from 
\cite{BYP}. 

\begin{definition} 
A continuous theory $T$ is called {\em decidable} 
if for every sentence $\phi$ the degree of truth  
$$ 
\phi^{\circ} = \mathsf{sup} \{ \phi^{M} :   M \models T\}
$$ 
is a computable real which is uniformly computable from $\phi$. 
\end{definition} 
This exactly means that there is an algorithm which 
for every $\phi$ and a rational number $\delta$ 
finds a rational $r$ such that 
$|r-\phi^{\circ}|\le \delta$. 

Note that decidability of $T$ does not imply that 
the set of all continuous $\phi$ with $\phi^M = 0$ 
for all $M\models T$, is computable 
(but for a complete $T$ this holds). 
On the other hand it is easy to see  that decidability 
of $T$ follows from this condition.  
This is a part of the following lemma.

\begin{lem} \label{q^o}
Let $T$ be a continuous theory in a computable language. 
Let a rational number $q^{\circ}$ belong to $[0,1]$. 

1. Assume that $q^{\circ} <1$ and there is an algorithm 
which decides for every formula $\phi$ without free variables 
whether $\phi^{\circ} \le q^{\circ}$. 
Then the theory $T$ is decidable. 

2. Assume that $q^{\circ} >0$ and there is an algorithm 
which decides for every formula $\phi$ without free variables 
whether $\phi^{\circ}$ equals $q^{\circ}$. 
Then the theory $T$ is decidable. 
\end{lem} 

\begin{proof} 
We start with the observation that the assumption of statement 1 
with any  $q^{\circ} <1$ is equivalent to the case $q^{\circ} =0$. 
This follows from the equivalence 
$$
\phi^{\circ} \le q^{\circ} \Leftrightarrow (\phi \dot{-}q^{\circ} )^{\circ} \le 0 . 
$$ 
In the case of statement 2 the following equivalence 
$$
\phi^{\circ} = q^{\circ} \Leftrightarrow (\frac{1}{q^{\circ}} \phi )^{\circ} =1  
$$
shows that the assumption of statement 2 
with any  $q^{\circ} >0$ is equivalent to the case $q^{\circ} =1$. 

To prove decidability of $T$ in the case of statement 1 
assume that $q^{\circ} = 0$. 
Given $\phi$ and $m>0$ find the minimal $\frac{t}{2m}$ 
so that $T\models \phi \dot{-} \frac{t}{2m} \le 0$. 
This defines an interval of the form 
$[\frac{s}{m} , \frac{s+1}{m} ]$ which contains 
$\phi^{\circ}$. 
  
In the case of statement 2 given $\phi$ and $m>0$  
find the minimal $\frac{t}{2m}$ so that 
$(\phi + \frac{t}{2m})^{\circ} =1$ with respect to $T$. 
This defines an interval of the form 
$[\frac{s}{m} , \frac{s+1}{m} ]$ which contains 
$\phi^{\circ}$.   
\end{proof} 

\begin{remark} 
 Lemma \ref{q^o} will be applied in Section 3 
in the situation when the segment $[0,1]$ is 
replaced by $[0,2]$. 
It obviously holds under the replacement $1$ by $2$ 
in the formulation.  
\end{remark} 

\subsection{Ershov's theorem} 

The following theorem is a counterpart of Ershov's 
decidability criterion (Theorem 6.1.1 of \cite{ershov}). 
Here we call a sequence of complete continuous theories $\{T_i, i\in\omega \}$  
{\em effective} if the relation   
$$ 
\{ (\theta ,j): \theta \mbox{ is a statement so that } T_j \vdash \theta \}
$$ 
is computably enumerable.    

\begin{thrm} \label{Ersov}
A continuous theory $T$ is decidable if and only if 
$T$ can be defined by a computably enumerable system of 
axioms and $T$ can be presented $T=\bigcap_{i\in\omega} T_i$  
where $\{T_i, i\in\omega \}$ is an effective sequence of 
complete continuous theories.    
\end{thrm} 
 
\begin{proof}  Sufficiency.  
Let $\phi$ be a continuous sentence. 
For every natural $n$ we can apply an effective procedure
which looks for conditions of the form $\phi \le \frac{k}{n}$ 
derived from the axioms of $T$ and conditions of 
the form $\frac{l}{n} \le \phi$ which appear in 
some $T_j \vdash \frac{l}{n} \le \phi$. 
By Corollary 9.8 from \cite{BYP} this always gives a number 
$k<n-1$ such that $\frac{k}{n} \le \phi^{\circ} \le \frac{k+2}{n}$.  

Necessity.
For every sentence $\phi$ we fix a computably enumerable 
sequence of segments $[l_{n,\phi}, r_{n,\phi}]$ converging 
to $\phi^{\circ}$ so that $\phi^{\circ} \in [l_{n,\phi}, r_{n,\phi}]$. 
Then all statements $\phi\le r_{n,\phi}$ form a computably 
enumerable sequence of axioms of $T$.

Now for every sentence $\phi$ we effectively 
build a complete theory $T_{n,\phi}\supset T$ with 
$T_{n,\phi} \vdash l_{n,\phi} \dot{-} \phi \le 2^{-n}$. 
In fact such a construction produces 
an effective family $T_i$, $i\in \omega$, 
from the formulation. 
Indeed, then for every natural $n$ 
we can find a sufficiently large $m$ so that 
$T_{m,\phi} \vdash \phi^{\circ} \dot{-} \phi \le 2^{-n}$ 
(here $\phi^{\circ}$ is defined by $T$). 
This obviously implies that 
$T$ coincides with the intersection 
of all $T_{m,\phi}$. 
Effectiveness will be verified below. 

At Step 0 for every $n$ we define $T_{n,\phi ,0}$ 
to be the extension of $T$ by the axiom 
$l_{n,\phi} \dot{-} \phi \le 0$.  
At every step $m+1$ we build 
a finite extension $T_{n,\phi ,m+1}$ 
of $T$ so that each inequality 
$\psi \le 0$ from $T_{n,\phi ,m} \setminus T$ 
is transformed into an inequality 
$\psi \le \varepsilon$, where 
$\varepsilon \le 2^{-(2n+m+1)}$. 
At later steps we consider these $\psi \le \varepsilon$ 
in the form $\psi \dot{-} \varepsilon \le 0$, i.e. 
the next transformation of them gives 
inequalities $\psi \dot{-} \varepsilon \le \varepsilon'$ 
(resp. $\psi \dot{-} (\varepsilon + \varepsilon') \le 0$). 
In this situation we say that the original $\psi \le 0$ is 
transformed into $\psi \le \varepsilon_1$, 
where $\varepsilon_1 = \varepsilon + \varepsilon'$. 
The 'limit theory' 
$T_{n, \phi} = \lim_{m\rightarrow \infty} T_{n,\phi, m}$ 
is defined by the limits of 
these values $\varepsilon , \varepsilon_1 , \ldots$  
for all formulas $\psi$.  
Note that it can happen that $\varepsilon \le 0$, 
i.e. the transformed inequality 
is of the form $\psi+\delta \le 0$, 
with $\delta>0$. 
On the other hand we will see that 
for every $\psi$ the axioms of  
$\lim_{m\rightarrow \infty} T_{n,\phi, m}$ 
give an effective sequence of rational numbers 
which converges to the value of $\psi$ 
under this theory. 

Let us enumerate all triples $(n,\phi ,\psi )$ 
by natural numbers $>0$ so that each triple 
has infinitely many numbers. 
Assume that the number $m+1$ codes a triple $(n, \phi , \psi )$. 
For all $n'\not= n$ we put $T_{n', \phi' ,m+1} = T_{n', \phi' ,m}$.  
Assume that at Step $m$ the theory 
$T_{n,\phi ,m}\setminus T$ already contains 
inequalities $\frac{k_l}{l} \le \psi_l \le \frac{k'_l }{l}$ 
for some natural $l$ and $k_l , k'_l \le l$.  
We admit that the $0$-th inequality 
$l_{n,\phi} \dot{-} \phi \le 0$ has been 
already transformed into an inequality 
$l_{n,\phi} \dot{-} \phi \le \varepsilon$  
for some $\varepsilon \le \sum_{i\le m} 2^{-(2n+i)}$. 
It appears as one of the inequalities $\psi_l \le \frac{k'_l}{l}$. 
Let $\theta$  be 
$$
\psi \dot{-} 2^{2n+m+1} \mathsf{max}_l ( \mathsf{max}(\psi_l \dot{-} \frac{k'_l }{l}, 
\frac{k_l}{l} \dot{-} \psi_l )) 
. 
$$ 
Since $T$ is decidable we compute 
$k_{m+1} <m$ so that 
$\frac{k_{m+1}}{m+1} \le \theta^{\circ} \le \frac{k_{m+1}+2}{m+1}$.  
Then the value of $\psi$ under $T_{n,\phi ,m}$ 
is equal to the value of $\theta$ under this theory and 
is not greater than $\frac{k_{m+1}+2}{m+1}$. 
This means that extending $T_{n,\phi ,m}$ by 
$0\le \psi \le \frac{k_{m+1}+2}{m+1}$ 
we preserve consistency of the theory.
If $k_{m+1} =0$ this finishes our construction at this step. 

If $k_{m+1}>0$ we need an additional 
correction. 
Let $\theta'$  be 
$$
\psi \dot{-} 2^{2n+m+1} \mathsf{max} ( \mathsf{max}_l (\mathsf{max} (\psi_l \dot{-} \frac{k'_l }{l}, 
\frac{k_l}{l} \dot{-} \psi_l )),  
\psi \dot{-} \frac{k_{m+1} +2}{m+1}). 
$$ 
Since $T$ is decidable we compute 
$k'_{m+1} <m$ so that 
$\frac{k'_{m+1}}{m+1} \le (\theta' )^{\circ} \le \frac{k'_{m+1}+2}{m+1}$.  
Then the value of $\psi$ under the extension of $T_{n,\phi ,m}$ 
by $\psi \le \frac{k_{m+1} +2}{m+1}$
is not greater than $\frac{k'_{m+1}+2}{m+1}$. 
This means that extending $T_{n,\phi ,m}$ by 
$\psi \le \frac{\mathsf{min}(k_{m+1}, k'_{m+1})+2}{m+1}$ 
we preserve consistency of the theory.

If $0< k'_{m+1} <k_{m+1}$ 
we repeat this construction again. 
It is clear that finally we arrive at the situation 
when after such a repetition the number $k_{m+1}$ 
does not change (or becomes $0$).  

Note that if the final $k_{m+1}$ is not equal to $0$, then 
the extension of $T$ by  \\  
$\psi \le  \frac{k_{m+1}+2}{m+1} + 2^{-(2n+m+1)}$ 
and all statements of the form 
$$
\frac{k_l}{l} -2^{-(2n+m+1)}  \le \psi_l \le \frac{k'_l }{l} +2^{-(2n+m+1)}  
\mbox{ (for inequalities } \frac{k_l}{l} \le \psi_l \le \frac{k'_l}{l}
\mbox{  from } T_{n,\phi ,m} ) 
$$ 
is consistent and the value $\psi^{\circ}$ with respect to this extension 
satisfies  $\frac{k_{m+1}}{m+1}\le \psi^{\circ}$. 
Indeed, since for the final version of $\theta$ 
(corresponding to the final $k_{m+1}$) we have 
$\frac{k_{m+1}}{m+1} \le \theta^{\circ}$ with respect to $T$, 
the following inequality must hold in any model of $T$ where $\theta$ 
takes the value $\theta^{\circ}$: 
$$
2^{2n+m+1} \mathsf{max} ( \mathsf{max}_l (\mathsf{max} (\psi_l \dot{-} \frac{k'_l }{l}, 
\frac{k_l}{l} \dot{-} \psi_l )),  
\psi \dot{-} \frac{k_{m+1} +2}{m+1}) <1 .  
$$ 
Thus the inequality  
$\psi \le  \frac{k_{m+1}+2}{m+1} + 2^{-(2n+m+1)}$ 
and the corresponding inequalities 
$$
\frac{k_l}{l} -2^{-(2n+m+1)}  \le \psi_l \le \frac{k'_l }{l} +2^{-(2n+m+1)}   
$$ 
are satisfied  
in any model of $T$ where $\theta$ takes the value 
$\theta^{\circ}$. 
Since $\theta^{\circ}\le \psi^{\circ}$, we have 
the latter inequality above. 

We now define $T_{n, \phi ,m+1}$ 
as the set of so corrected statements of $T_{n,\phi ,m}$ 
together with the statement 
$$
\frac{k_{m+1}}{m+1}\le \psi \le \frac{k_{m+1}+2}{m+1}+ 2^{-(2n+m+1)} . 
$$ 
If $\psi$ also occurs as some $\psi_l$ above then 
we obviously add the strongest inequalities to $T_{n, \phi ,m+1}$.  
By the argument of the previous paragraph 
the obtained extension is consistent with $T$.

By the choice of a repeating enumeration
we see that for each sentence $\psi$ 
boundaries of $\psi$ at steps of our procedure 
form a Cauchy sequences with the same limit. 
Thus $\psi$ has the same value in all models of 
$T_{n,\phi}$. 
Moreover the inequality $l_{n,\phi} \dot{-} \phi \le 0$ 
will be transformed into $l_{n,\phi} \dot{-} \phi \le 2^{-n}$. 
We see that Step 0 guarantees that $T$ coincides 
with the intersection  of all $T_{n,\phi}$.  

Note that after the $(m+1)$-th step 
we know that for every inequality $\psi'\le \delta$ 
from each $T_{n',\phi' ,m+1}\setminus T$ 
the upper boundary of $\psi'$ in the final 
$T_{n' ,\phi'}$ cannot exceed $\delta + \frac{1}{2^{m}}$. 
In particular all inequalities of this kind can be included 
into an enumeration of axioms of $T_{n',\phi'}$ at this step.   
Thus we see that by the effectiveness of our procedure 
the family $\{ T_{n,\phi} \}$ is effective.  
\end{proof} 

\subsection{Interpretability} 

In order to have a method for 
proving undecidability of continuous 
theories we now discuss 
interpretability of first order structures 
in continuous ones.  

Let $L_0 =\langle P_1 ,...,P_m  \rangle$ be a finite 
relational signature.  
Let $\mathcal{K}_0$ be a class 
of finite first-order $L_0$-structures. 
Let $\mathcal{K}$ be a class of continuous 
$L$-structures, where $L$ is as above. 
We say that $\mathcal{K}_0$ is {\em relatively 
interpretable} in $\mathcal{K}$ if there is a finite constant 
extension $L(\bar{a}) = L\cup \{ a_1 ,...,a_r \}$, 
a constant expansion $\mathcal{K} (\bar{a})$ of $\mathcal{K}$ 
(we admit the situation that $\bar{a}$ is empty) 
and there  are continuous $L$-formulas 
$$ 
\phi^{-} (\bar{x}, \bar{y}) \mbox{ , } \phi^{+} (\bar{x}, \bar{y}) 
\mbox{ , } \theta^{-} (\bar{x}, \bar{y}_1 , \bar{y}_2) 
\mbox{ , } \theta^{+} (\bar{x}, \bar{y}_1 , \bar{y}_2)  \mbox{ and } 
$$ 
$$ 
\psi^{-}_1 (\bar{x}, \bar{y}_1 , \bar{y}_2 ,...,\bar{y}_{l_1} ) 
\mbox{ , } \psi^{+}_1 (\bar{x}, \bar{y}_1 , \bar{y}_2 ,...,,\bar{y}_{l_1} ) \mbox{  , ..., } 
\psi^{-}_m (\bar{x}, \bar{y}_1 , \bar{y}_2 ,...,\bar{y}_{l_m} ) 
\mbox{ , } \psi^{+}_m (\bar{x}, \bar{y}_1 , \bar{y}_2 ,...,,\bar{y}_{l_m} ) , 
$$ 
$$ 
\mbox{ with  } 
|\bar{y} |= |\bar{y}_1 | = |\bar{y}_2 | = ... |\bar{y}_{l_j}| = ... =|\bar{y}_{l_m} | 
\mbox{ , such  that: } 
$$ 
(i) the $L$-reduct of $\mathcal{K}(\bar{a})$ coincides with $\mathcal{K}$;\\ 
(ii) the conditions $\phi^{-} (\bar{a}, \bar{y})\le 0$ and 
$\phi^{+} (\bar{a}, \bar{y}) > 0$ are equivalent  in any $M\in \mathcal{K}(\bar{a})$
and  the condition $\theta^{-} (\bar{a}, \bar{y}_1, \bar{y}_2)\le 0$  
defines an equivalence relation on the zero-set of $\phi^{-} (\bar{a} ,\bar{y})$ 
(on tuples of the corresponding power $M^s$ with $s = |\bar{y}_1 |$),  
so that the values of any 
$\psi^{\varepsilon}_i (\bar{a},\bar{y}_1 , \bar{y}_2 ,...,\bar{y}_{l_i} )$ 
are invariant under this equivalence relation; \\ 
(iii)  the $(+)$-conditions below are equivalent to $(-)$-ones in $\mathcal{K}(\bar{a})$ :  
$$ 
\theta^{-} (\bar{a}, \bar{y}_1, \bar{y}_2)\le 0 \mbox{ , }  
\theta^{+} (\bar{a}, \bar{y}_1 ,\bar{y}_2) > 0 \mbox{ , } 
\psi^{-}_1 (\bar{a}, \bar{y}_1 , \bar{y}_2 ,...,\bar{y}_{l_1} ) \le 0 
\mbox{ , } 
$$ 
$$ 
\psi^{+}_1 (\bar{a}, \bar{y}_1 , \bar{y}_2 ,...,,\bar{y}_{l_1} )>0 \mbox{  , ..., } 
\psi^{-}_m (\bar{a}, \bar{y}_1 , \bar{y}_2 ,...,\bar{y}_{l_m} ) \le 0
\mbox{ , } 
\psi^{+}_m (\bar{a}, \bar{y}_1 , \bar{y}_2 ,...,,\bar{y}_{l_m} )>0 ;  
$$ 
(iv) for any $M \in \mathcal{K}(\bar{a})$ the conditions of (iii)  define 
an $L_0$-structure from $\mathcal{K}_0$ on the $\theta$-quotient   
of the zero-set  of $\phi^{-} (\bar{a} ,\bar{y})$  and 
any structure of $\mathcal{K}_0$ can be so realized.  

\begin{thrm} \label{unde}  
Assume that the class of finite structures $\mathcal{K}_0$ 
is relatively interpretable in $\mathcal{K}$
and assume that $Th(\mathcal{K}_0 )$ is undecidable. 
Then  the continuous theory  $Th(\mathcal{K}(\bar{a}))$ 
of the corresponding constant expansion  
is not a computable set. 
\end{thrm} 

\begin{proof} 
The proof is straightforward. 
To each formula $\psi$ of the theory of $\mathcal{K}_0$  so that 
the quantifier-free part is in the disjunctive normal form we associate 
the appropriately rewritten continuous formula $\psi^- (\bar{a}, \bar{z})$ 
(with appropriate free variables) and the $0$-statement $\psi^- (\bar{a}, \bar{z}) \le 0$.  
In particular atomic formulas are written by $(-)$-conditions 
above, but negations of atomic formulas 
appear in the form of 
$$
\psi^{+}_i (\bar{a}, \bar{y}_1 , \bar{y}_2 ,...,,\bar{y}_{l_i} )\le 0 . 
$$  
Condition (ii) and the condition that  the $\theta$-quotient   
of the zero-set  of $\phi^{-} (\bar{a} ,\bar{y})$ 
is always finite, allow us to use standard quantifiers 
in such statements $\psi^- (\bar{a}, \bar{z})\le 0$: 
the quantifier $\forall$ 
is written as $sup$ but $\exists$ is written as $inf$.  

Note that if $\psi'$ is equivalent to $\neg \psi$  
then  $(\psi' )^{-} (\bar{a},\bar{z})\le 0$ is equivalent 
to  $\psi ^{-} (\bar{a}, \bar{z})> 0$ for tuples from 
the zero-set of $\phi^{-} (\bar{a}, \bar{y})$ 
(and $\psi^- (\bar{a}, \bar{z})>0$ is equivalent to 
the corresponding $\psi^+ (\bar{a}, \bar{z})\le 0$). 

It is easy to see that this construction reduces 
the decision problem for $Th(\mathcal{K}_0 )$ 
to  computability of the set 
$Th(\mathcal{K}(\bar{a}))$.  
\end{proof}

This theorem will be applied in Section 5 under 
circumstances that $\mathcal{K}(\bar{a})= \mathcal{K}$. 

It is worth noting that the theorem gives 
a relatively weak method of proving undecidability 
of continuous theories. 
In the classical first-order logic 
such a situation usually has much stronger consequences. 
For example Theorem 5.1.2 of \cite{ershov} in a slightly modified 
setting (and removing the assumption that $\mathcal{K}_0$ 
consists of finite structures) states that hereditary 
undecidability of $Th(\mathcal{K}_0)$ can 
be lifted to $Th(\mathcal{K})$.   
In the following remarks we describe several difficulties 
arising in our approach. 

\begin{remark} 
As we already know the statement of Theorem \ref{unde} 
does not imply that $Th(\mathcal{K}(\bar{a}))$ is undecidable. 
It seems to us that it is a challenge 
to find a useful method of interpretability which gives undecidability 
of the theory. 
\end{remark} 

\begin{remark} 
The assumption that $\mathcal{K}_0$ consists of finite structures is essential (see the proof of Theorem \ref{unde}). 
The 'positiveness' of the continuous logic does not allow 
stronger statements. 
\end{remark} 

\begin{remark} 
Assuming that $Th(\mathcal{K}_0 )$ is not stable we cannot 
state the same for the theory of $\mathcal{K}(\bar{a})$.  
This follows from the requirement that 
in the definition of the order property for a sequence 
$\bar{a}_1, \ldots , \bar{a}_k ,\ldots$ 
the inequality  $\phi (\bar{a}_i ,\bar{a}_j ) \not= 0$ (when $i\ge j$) 
implies $\phi (\bar{a}_i ,\bar{a}_j ) = 1$ 
(see Section 5 of \cite{FHS}).    
\end{remark}

\section{Decidability of theories of pseudo finite dimensional Hilbert spaces }

We start this section with the observation that the theory 
of all finite dimensional dynamical Hilbert spaces is decidable 
if it is computably axiomatizable (see Theorem \ref{Fin-dim}).  
Connections of decidability and pseudocompactness 
with property MF are discussed in Section 4.2. 
In Section 4.3 it is proved that the universal theory of 
(finite dimensional)  dynamical Hilbert spaces is decidable. 
In Section 4.4 we consider the problem 
when a dynamical Hilbert space is pseudo finite dimentional.

\subsection{Finite dimension} 

Let us now restrict the dimension 
of Hilbert spaces, say by $N$. 
It is natural to expect that 
then the theory of (dynamical) 
Hilbert spaces becomes decidable. 
In classical model theory this corresponds 
to the situation of a theory of structures 
of a fixed finite size.

Let us fix a signature 
$$
(\{ B_n\}_{n\in \omega} ,0,\{ I_{mn} \}_{m<n} ,
\{ \lambda_c \}_{c\in\mathbb{Q}[i]}, +,-,\langle \rangle_{Re} , \langle \rangle_{Im}, U_1 ,...,U_t  ),
$$
where as before we assume that $U_j$, $1\le j \le t$,  
are  symbols of unitary operators of $\mathbb{H}$ 
which are defined only on $B_1$. 
Using Theorem \ref{Ersov} we will 
prove that the theory of $N$-dimensional 
spaces in this language is decidable.

\begin{remark} 
On the other hand  since the structures 
are of infinite language it is not very difficult 
to find such a structure 
with undecidable continuous theory. 
For example one can take a dynamical 3-dimensional 
Hilbert space with an additional operator $U$ such that 
$\mathsf{sup}_{v\in B_1}  d(v, U(v))=r$, 
where  $r$ is a non-computable real number which belongs to $[0,2]$. 
\end{remark}

Let us enumerate all $N$-dimensional 
unitary matrices of algebraic complex numbers. 
This can be arranged by some canonical indexing 
of all algebraic numbers (for example see \cite{LM}) 
and using decidability of the theory of algebraically closed fields. 
This induces an enumeration $Axm_j$, $j\in \omega$, 
of systems of axioms of complete continuous theories  
$T_j$ of dynamical $N$-dimensional spaces. 
Each $Axm_j$  consists of the standard 
axioms of $N$-dimensional spaces, the axioms stating that each 
$U_s$ is a unitary operator and the axioms 
describing the matrices of all $U_s$ in some  
basis: 
$$ 
\mathsf{inf}_{y_1 ,...,y_N } \mathsf{max}
(\mathsf{max}_{1\le i\le N} (|\langle y_i ,y_i\rangle -1|), 
$$
$$
\mathsf{sup}_x (| ( \langle x,x\rangle  - |\langle x,y_1 \rangle |^2 - ... 
...-|\langle x, y_N \rangle |^2 ) |), 
$$
$$ 
\mathsf{max}_{1\le l\le N} \mathsf{max}_{1\le j\le t} (\parallel U_j (y_l )-\sum \lambda_{c_{j,l,k}} (y_k )  \parallel \dot{-} \varepsilon_l ))\le 0 , 
$$ 
$$
\mbox{ where } \varepsilon_l \in \mathbb{Q} 
\mbox{ and } c_{j,l,k} \in \mathbb{Q}[i] 
\mbox{ are appropriate approximations } 
$$ 
$$ 
\mbox{ of entries of matrices for } 
U_1 ,\ldots ,U_t . 
$$ 

Using Lemma \ref{comp} it is easy to see that each $Axm_j$ 
axiomatizes a decidable theory and the enumeration 
$Axm_j$, $j\in \omega$, 
gives an effective indexation of complete continuous theories  
$T_i$ of dynamical $N$-dimensional spaces in the sense of Section 2.    
The statement that the relation 
$\{ (\theta ,j): \theta$ is a statement so that $T_j \vdash \theta \}$ 
is computably enumerable follows from   
the fact that this relation coincides with 
$\{ (\theta ,j): \theta$ is a statement so that $Axm_j \vdash \theta \}$.   

\begin{thrm} \label{N-dim} 
The theory of all dynamical $N$-dimensional Hilbert spaces 
with operators $U_1 ,\ldots , U_t$ 
coincides with the intersection $\bigcap T_j$.  \\ 
The theory of all dynamical $N$-dimensional Hilbert spaces is decidable. 
\end{thrm} 

\begin{proof} 
As we already know the theory of all dynamical 
$N$-dimensional spaces is finitely axiomatizable. 
Thus by Theorem \ref{Ersov} the second statement of 
the theorem follows from the first one. 
To prove it we only have to show that 
for any rational $\delta$, any dynamical $N$-dimensional space  
$$
(\{ B_n\}_{n\in \omega} ,0,\{ I_{mn} \}_{m<n} ,
\{ \lambda_c \}_{c\in\mathbb{Q}[i]}, +,-,\langle \rangle_{Re} , \langle \rangle_{Im}, U_1 ,\ldots ,U_t  ),
$$ 
and any continuous sentence $\theta ( U_1 ,\ldots ,U_t  )$ over this 
structure there are unitary operators $\tilde{U}_1 ,\ldots ,\tilde{U}_t$ 
defined by matrices over $\mathbb{Q}[i]$, so that  
$$ 
| \theta ( U_1 ,\ldots ,U_t  ) - \theta ( \tilde{U}_1 ,\ldots ,\tilde{U}_t  )| \le \delta . 
$$ 
Indeed this shows that when some $\theta ( U_1 ,\ldots ,U_t  ) \le \varepsilon$ 
does not belong to $T$, then it does not belong 
to some $T_j$ (defined by matrices of $\tilde{U}_1 ,\ldots ,\tilde{U}_t$). 

Since any continuous formula defines a uniformly 
continuous function and the ball $B_1$ is compact 
it suffices to take $\tilde{U}_1 ,\ldots ,\tilde{U}_t$ so that 
they sufficiently approximate $U_1 ,\ldots ,U_t$.  
This is a folklore fact. 
On the other hand it is a curious place where 
the following fact from quantum computations can be applied 
(the information given in the beginning of Section 5 suffices 
for the terminology below). 

Let $\mathcal{B}$ be a 2-dimensional space over $\mathbb{C}$. 
Let $(\mathcal{B})^{\otimes 2}$ be the $4$-dimensional space with 
the (Dirac) basis 
$$ 
|00\rangle \mbox{ , } |01\rangle \mbox{ , } |10\rangle \mbox{ , } |11\rangle . 
$$ 
Let $CNOT$ be a 2-qubit linear operator on 
$(\mathcal{B})^{\otimes 2}$ defined by  
$$ 
CNOT: |00\rangle \rightarrow |00\rangle \mbox{ , } |01\rangle \rightarrow |01\rangle 
\mbox{ , } 
|10\rangle \rightarrow |11\rangle \mbox{ , } |11\rangle \rightarrow |10\rangle   
. 
$$ 
The {\em Toffoli gate} is a 3-qubit linear operator defined on basic vectors of $(\mathcal{B})^{\otimes 3}$ by 
$$
\Lambda(CNOT): |\varepsilon_1 \varepsilon_2 \varepsilon_3 \rangle \rightarrow 
|\varepsilon_1 \varepsilon_2 (\varepsilon_3 \oplus \varepsilon_1 \cdot \varepsilon_2 )\rangle \mbox{ , where } \varepsilon_1 , \varepsilon_2 ,\varepsilon_3 \in \{ 0,1\} .  
$$ 

It is well-known (see \cite{ksv}, Section 8) that   
\begin{quote} 
(a) For any natural number $k\ge 2$ all unitary transformations of 
$(\mathcal{B} )^{\otimes k}$ can be presented as products of  
1-qubit unitary transformations and  2-qubit copies of CNOT  at appropriate registers. 

(b)  The operators of the basis  
$$ 
\mathcal{Q}  = \{  K = {{ 1 \mbox{ } 0}\choose{0 \mbox{ } i}} , CNOT , \Lambda (CNOT), 
 \mbox{Hadamar's } H= \frac{1}{\sqrt{2}}{{ 1 \mbox{ } 1}\choose{1 \mbox{ } -1}} \} 
$$ 
generate a dense subgroup of  
$\mathbb{U}(\mathcal{B}^{\otimes 3})/\mathbb{U}(1)$ under the operator norm. 
\end{quote}

These facts reduce the problem of construction of 
$\tilde{U}_1 ,\ldots ,\tilde{U}_t$ to the case of dimension 2. 
The latter case follows from standard presentations of 
unitary $2\times 2$-matrices.  
\end{proof}

The method of this theorem can be easily adapted to the following 
statement. 

\begin{thrm} \label{Fin-dim} 
Assume that the theory $T_{f.d}$ of all finite dimensional dynamical 
Hilbert spaces of the signature 
$$  
(\{ B_l\}_{l\in \omega} ,0,\{ I_{kl} \}_{k<l} ,
\{ \lambda_c \}_{c\in\mathbb{Q}[i]}, +,-,\langle \rangle_{Re} , 
\langle \rangle_{Im}, U_1 ,U_2 ,\ldots ,U_t  )  
$$ 
is computably axiomatizable. 
Then it is decidable. 

In particular assume that any dynamical Hilbert space of this 
signature is elementarily equivalent to an ultraproduct of 
finite dimensional dynamical Hilbert spaces. 
Then the theory of all dynamical Hilbert spaces of this 
signature is decidable. 
\end{thrm}

\begin{proof} 
We modify the proof of Theorem \ref{N-dim} starting with 
enumeration of all {\em finite} dimensional 
unitary matrices of algebraic complex numbers.  
This induces an enumeration of systems of axioms 
of complete continuous theories $T^{fin}_j$ of dynamical 
$N$-dimensional spaces where $N$ is not fixed. 
The axioms describing the matrices of all $U_s$ in some  
basis are the same as before, where $N$ depends on the number 
of $T^{fin}_j$.  
This gives an effective indexation of complete 
continuous theories  $T^{fin}_j$ of dynamical 
finite dimensional spaces in the sense of Section 2.    
The proof that the theory $T_{f.d}$ 
coincides with the intersection $\bigcap T^{fin}_j$ is the same as 
in Theorem \ref{N-dim}. 
Now the first statement of the theorem follows from Theorem 
\ref{Ersov}. 

To see the second statement just note that the assumption of it 
says that the theory $T_{f.d}$ 
is axiomatizable by 
standard axioms of dynamical Hilbert spaces. 
\end{proof}

The crucial point of the theorem above is 
the assumption that the theory $T_{f.d}$ 
is recursively axiomatizable. 
We do not know if this holds. 
We will see in the following section that 
this question is connected with an open problem in the theory of 
approximations by metric subgroups. 

\begin{remark} 
It is a folklore fact that 
any Hilbert space (without operators) is 
elementarily equivalent to 
an ultraproduct of finite dimensional Hilbert spaces. 
Thus in the case when all $U_i$ are equal to the identity 
map, the argument above shows that 
the theory of all Hilbert spaces  
is decidable (which is also folklore). 
\end{remark}

\subsection{Unbounded dimension and property MF} 

In this section we find a connection between 
the assumptions of the second statement of 
Theorem \ref{Fin-dim} and the topic of 
approximations by metric groups.   
The latter is deserved a particular attention in group theory.  
This is mainly motivated by investigations of 
{\em sofic and hyperlinear groups}.  
We remind the reader that a group $G$ is called 
{\em sofic} if $G$ embeds into 
a metric ultraproduct of finite symmetric groups 
with the normalized Hamming distance $d_{H}$, \cite{pestov}:  
$$
d_H (g,h) = 1 - \frac{|\mathsf{Fix} (g^{-1}h)|}{n} \mbox{ for } 
g,h \in S_n . 
$$ 
A group $G$ is called {\em hyperlinear} if $G$ embeds into 
a metric ultraproduct of finite-dimensional unitary  groups  $U(n)$ 
with the normalized Hilbert-Schmidt metric $d_{HS}$ 
(i.e. the standard $l^2$ distance between matrices), \cite{ElSa},  \cite{pestov}. 
It is an open question whether these classes are the same 
and whether every countable group is sofic/hyperlinear. 

The use of metric ultraproducts can be 
replaced by the following notion 
of approximation, see  \cite{thom} and \cite{G} (Definition 3). 
In this definition and below we always assume that 
metric groups are considered with respect to invariant metrics. 

\begin{definition} \label{DefApp} 
Let $\mathcal{K}$ be a class of metric groups. 
We say that a group $G$ is $\mathcal{K}$-{\em approximable} if 
there is a function $\alpha : G\rightarrow [0,\infty ]$ with 
$$ 
\alpha(g) = 0 \Leftrightarrow g=1 , 
$$ 
so that for any finite $F\subset G$ and $\varepsilon >0$ there is 
$(H,d) \in \mathcal{K}$ and a function $\gamma : F \rightarrow H$ 
so that 
$$
\mbox{ if } 1 \in F \mbox{ then } d(1,\gamma (1)) <\varepsilon  \mbox{ , } 
$$
$$ 
\mbox{ for any } g,h, gh \in F \mbox{ , } d (\gamma (gh), (\gamma (g) \gamma(h) ) ) < \varepsilon  \mbox{  and } 
$$ 
$$ 
\mbox{ for any } g\in F \mbox{ , } d(1,\gamma (g))\ge \alpha (g).   
$$ 
\end{definition} 

It is known that when the metrics of $\mathcal{K}$ are bounded 
by some fixed number $r$, a group $G$ is $\mathcal{K}$-approximable 
if and only if it embeds into a metric ultraproduct of groups 
from $\mathcal{K}$ (\cite{thom} and \cite{G}). 
Moreover in the case of sofic and hyperlinear groups the function $\alpha$ 
can be taken constant on $G\setminus \{ 1 \}$ with the value 
equal to any real number strictly between $0$ and $1$ 
(between $0$ and $r=\sqrt{2}$ in the hyperlinear case).  
We develop this property of sofic and hyperlinear groups as follows. 

\begin{definition} \label{DefAmpl} 
Let $G$ be an abstract group, $\mathcal{K}$ be 
a class of metric groups and $\alpha_0$ be a function 
$G\rightarrow [0,\infty ]$ with $\alpha_0 (1) = 0$. 
Assume that $G$ is $\mathcal{K}$-approximable.  
We say that $\alpha_0$ is the {\em amplification bound} of $G$ 
with respect to $\mathcal{K}$ if for any $g\not=1$, 
$\alpha_0 (g)$ is the supremum of all possible values $\alpha (g)$ 
with respect to all possible function 
$\alpha : G \rightarrow [0, \infty)$ 
satisfying the properties of Definition \ref{DefApp}. 
\end{definition} 

Note that in the case of sofic groups the amplification bound 
with respect to the class of symmetric groups with normalized 
Hamming metrics is the function which is $1$ for all nontrivial elements.

Below instead of examples mentioned above we will consider 
the following one. 
\begin{quote} 
Unitary groups $U(n)$ together with 
the metric induced by the operator norm (on $G_L(n, \mathbb{C})$)
$\parallel T\parallel_{op} = \mathsf{sup}_{\parallel v\parallel = 1} \parallel Tv\parallel$.   \\ 
We put $d(T,Q) = \parallel T - Q\parallel_{op}$.  
\end{quote}  
This metric is {\em submultiplicative}, i.e. it is defined by a norm 
on $M_n (\mathbb{C})$ which satisfy the property 
$\parallel AB\parallel \le \parallel A\parallel \cdot \parallel B\parallel$. 

Groups approximable by these metric groups are called MF (matricial field), see \cite{CDE}. 
It is an open question if there are non-MF groups. 
A. Tikuisis, S. White and W. Winter proved in \cite{tikuisis} 
that amenable groups are MF. 
A. Korchagin shows in the recent preprint \cite{korchagin} that 
in many respects property MF is similar to soficity and hyperlinearity.   

\begin{remark} 
It is worth mentioning that another submultiplicative 
metric on $U(n)$ can be defined with respect 
to the Frobenius norm = the unnormalized 
Hilbert-Schmidt norm  
$\parallel T\parallel_{Frob} = \sqrt{\sum_{i,j} |T_{ij}|^2}$ 
(i.e. just the $l^2$-distance).  
In this case the corresponding groups are called  
{\em Frobenius approximated} {\em \cite{CGLTh}}.  
It is already proved in \cite{CGLTh} 
that there are finitely presented groups which are not Frobenius approximated. 
However there is no any description of the class of 
Frobenius approximated groups. 
\end{remark}

The following theorem is the most important observation of this section.

\begin{thrm} \label{hyplin} 
Let $G=\langle g_1 ,...,g_n \rangle$ 
be a finitely generated group.  
The group $G$ is MF if and only if 
there is a dynamical Hilbert space in the signature 
$$  
(\{ B_l\}_{l\in \omega} ,0,\{ I_{kl} \}_{k<l} ,
\{ \lambda_c \}_{c\in\mathbb{Q}[i]}, +,-,\langle \rangle_{Re} , 
\langle \rangle_{Im}, U_1 ,U_2 ,\ldots ,U_n  )  
$$ 
which is an ultraproduct of finite dimensional 
dynamical Hilbert spaces of the same signature and 
the group $\langle U_1 ,\ldots , U_n \rangle$ is 
isomorphic to $G$ under the map taking $U_i$ to $g_i$, 
$1\le i \le n$. 
\end{thrm} 

\begin{proof} 
Below use $d$ both for metrics in Hilbert spaces and 
for metrics of approximating metric groups. 

Sufficiency of the theorem is easy. 
Indeed, having a dynamical Hilbert space (say ${\bf H}$)
as in the statement consider the family of 
finite dimensional dynamical Hilbert 
spaces occurring in the corresponding ultraproduct. 
To define the function $\alpha$ from Definition \ref{DefApp} 
for any $g\in G\setminus \{ 1\}$ just 
take $\alpha (g)$ to be a positive real number 
which is less than 
$\mathsf{sup}_{v\in B_1} d(v, g(v))$ 
computed in ${\bf H}$.  

Since the inequalities of Definition  \ref{DefApp} 
in the case of the operator norm can be 
written by formulas of continuous logic, 
the approximations which we need in this definition 
can be taken as groups generated by 
$U_1 ,U_2 ,\ldots ,U_n$  in spaces of the family from 
the ultraproduct. 
Then the function $\gamma$ appearing in such 
an approximation maps a word of $F$ to 
the corresponding word written in 
$U_1 ,U_2 ,\ldots ,U_n$.  
For an illustration we give a formula for  the condition 
$$
\mbox{ if } g \in F \mbox{ then } d(1,\gamma (g)) \ge \alpha (g).  
$$
Assume that $g$ is presented by a word $w(g_1 ,\ldots ,g_n )$. 
Then we formalize the condition above as follows. 
$$
\alpha (g) \dot{-} \mathsf{sup}_{v\in B_1} d(v, w(U_1 ,\ldots ,U_n )v) \le 0. 
$$

Let us prove the necessity of the theorem. 
Let $m>0$ and let $F\subseteq G$ be the ball 
of elements of $G$ presented by words of length $\le m$. 
Let $\varepsilon$ be a small real number.  
Since $G$ is MF there is 
an embedding $\gamma$ of $F$ into some $U(l)$ 
which satisfies the conditions of Definition \ref{DefApp} 
for the corresponding metric. 
We may assume that the corresponding function $\alpha$ 
is greater than $|F|\varepsilon$ for non-trivial elements of $F$.  

Let $w(x_1 ,\ldots ,x_n )$ be a word of length $\le m$.  
If we present this word in the form 
$(\ldots (x^{\delta_1}_{i_1} x^{\delta_2}_{i_2}) \ldots ) x^{\delta_m}_{i_m}$ 
with $\delta_i \in \{ -1,0, 1\}$, then we have 
$$ 
d (\gamma (g^{\delta_1}_{i_1}) \gamma (g^{\delta_2}_{i_2}), \gamma (g^{\delta_1}_{i_1} g^{\delta_2}_{i_2}))\le \varepsilon   
\mbox{ , } d (\gamma (g^{\delta_1}_{i_1} g^{\delta_2}_{i_2}) \gamma (g^{\delta_3}_{i_3}), \gamma (g^{\delta_1}_{i_1} g^{\delta_2}_{i_2} g^{\delta_2}_{i_3}))\le \varepsilon  \mbox{ , } \ldots
$$
By invariantness of $d$ this implies that 
$$ 
d ((\gamma (g^{\delta_1}_{i_1}) \gamma (g^{\delta_2}_{i_2}))\gamma (g^{\delta_3}_{i_3}),  \gamma (g^{\delta_1}_{i_1} g^{\delta_2}_{i_2} g^{\delta_2}_{i_3}))\le 2\varepsilon  \mbox{ , } 
$$ 
$$
d (\gamma (g^{\delta_1}_{i_1}) \gamma (g^{\delta_2}_{i_2}) \gamma (g^{\delta_3}_{i_3}) \gamma (g^{\delta_4}_{i_4}), \gamma (g^{\delta_1}_{i_1} g^{\delta_2}_{i_2} g^{\delta_3}_{i_3} g^{\delta_4}_{i_4}))\le 3\varepsilon  \mbox{ , } \ldots .
$$
As a result we see that 
$d(\gamma(w(\bar{g}) ), w(\overline{\gamma(g)})) \le (m-1)\varepsilon$. 
In particular if $G\models w(\bar{g})\not= 1$, then 
$d(1,\gamma (w(\bar{g} )))\ge min (\alpha (F))$ and 
$d(1, w(\overline{\gamma(g)})) \ge min (\alpha (F)) - (m-1)\varepsilon$. 
On the other hand if $G\models w(\bar{g})= 1$, then 
$d(1,\gamma (w(\bar{g} )))\le \varepsilon$ and 
$d(1, w(\overline{\gamma(g)})) \le m\varepsilon$.

Let $M_{\varepsilon, F}$ be the corresponding finite dimensional dynamical  
Hilbert space: 
$$  
(\{ B^{ H}_n\}_{n\in \omega} ,0,\{ I_{mn} \}_{m<n} ,
\{ \lambda_c \}_{c\in\mathbb{Q}[i]}, +,-,\langle \rangle_{Re} , 
\langle \rangle_{Im}, \gamma (g_1 ) ,\gamma (g_2 ),\ldots ,\gamma (g_n ) ).  
$$ 
The computations above show that for any $v\in B_1$ of norm 1 
the distance \\ 
$d(\gamma(w(\bar{g}) )v, w(\overline{\gamma(g)})v)$ 
is not greater than  $\le (m-1)\varepsilon$. 

Let us fix an enumeration of pairs $(\varepsilon_i ,F_i )$, $i\in \omega$, 
as above with $\varepsilon_i \rightarrow 0$ and 
$G= \bigcup F_i$. 
Let $D$ be a non-principal ultrafilter on $\omega$. 
We assume that $\varepsilon_i > |F_{i+1}|\varepsilon_{i+1}$ and 
$F_i \subset F_{i+1}$. 
Let us prove that in the corresponding $D$-ultraproduct of 
the  structures $M_{\varepsilon_{i}, F_{i}}$ 
the  tuple $U_1 ,\ldots U_n$ corresponding 
to $g_1 , \ldots , g_n$, generates  
a group naturally isomorphic to $G$.

Let $m$ be a natural number and 
$w(x_1 ,\ldots ,x_n )$ be a word of length $\le m$. 
Assume that $G\models w(g_1 ,\ldots ,g_n )=1$. 
As we have shown above for any $\varepsilon > 0$ 
there is a member of the sequence  
$(\varepsilon_i ,F_i )$, $i\in \omega$, 
such that for all numbers after this pair 
the statement   
$$
\mathsf{sup}_{v\in B_1} d(v, w(\overline{\gamma_i (g)})v) \le \varepsilon 
$$ 
holds in the corresponding 
structures $M_{\varepsilon_i , F_i }$ (for appropriate $\gamma_i$ ). 

If $w(g_1 ,\ldots ,g_n )$ is not equal to 1, then 
there is a rational number $q$ (sufficiently close to 
$\alpha (w(\overline{\gamma_i (g)}))$ ) such that 
almost all structures $M_{\varepsilon_i , F_i }$ 
satisfy the statement   
$$
q \dot{-} \mathsf{sup}_{v\in B_1} d(v, w(\overline{\gamma_i (g)})v) \le 0. 
$$ 
The rest is clear. 
\end{proof}

Theorem \ref{hyplin} implies that 
the statement that any finitely generated group is MF 
(which is a well-known conjecture)
follows from the statement that the regular representation of 
any finitely generated group is pseudo finite dimensional. 
We will discuss the latter statement in Section 4.4. 

Although the following theorem is not 
absolute, the assumptions of it are satisfied if 
every countable group is MF 
(which is a well-known conjecture).
Indeed a finitely presented group with 
undecidable word problem was constructed by 
Novikov in the 50-s, see \cite{novikov}.  

\begin{thrm} \label{Novikov}
Assume that there is an MF finitely 
presented group $G= \langle g_1 ,...,g_n |\mathcal{ R}\rangle$ 
with undecidable word problem. 
Let $T_G$ be the theory of the signature 
$$  
( \{ B_n\}_{n\in \omega} ,0,\{ I_{mn} \}_{m<n} ,
\{ \lambda_c \}_{c\in\mathbb{Q}[i]}, +,-,\langle \rangle_{Re} , 
\langle \rangle_{Im}, U_1 ,U_2 ,\ldots ,U_n  )  
$$ 
axiomatized by all statements satisfied in all 
finite dimensional dynamical Hilbert spaces and the statements    
$$
\mathsf{sup}_{v\in B_1} d(v, w(\overline{U})v) \le 0 \mbox{ , where } 
w(\bar{g})\in \mathcal{ R} . 
$$
Then the set of statements of $T_G$ is not decidable. 
\end{thrm} 

Before the proof we give two remarks. 

\begin{remark} 
In fact in the formulation of the theorem we use 
the conventions of Section 2. 
In particular we extend the signature by symbols 
$U'_{i}$ for $U^{-1}_i$, $i\in \omega$, 
and also add axioms 
$\mathsf{sup}_{v\in B_1} d(v, U'_{i} (U_{i}(v))) \le 0$ 
and  $\mathsf{sup}_{v\in B_1} d(v, U_i (U'_{i}(v))) \le 0$. 
The $\mathsf{sup}$-formulas in the formulation are 
written with $U'_i$ for $U^{-1}_i$. 
\end{remark} 

\begin{remark} 
If we do not assume in the formulation that 
"statements are satisfied in {\em finite dimensional}" 
members of $\mathcal{ K}$, then the theorem becomes much easier. 
The proof is  basically the same as the proof below 
but does not use Theorem  \ref{hyplin}. 
We just use the (infinitely dimensional) Hilbert   
space ${\it l}^2 (G)$. 
\end{remark} 

\bigskip 

\begin{proof} ({\em Theorem \ref{Novikov}}) 
The idea of this proof is well-known. 
W. Baur was the first who applied it, see \cite{baur}. 
Let 
$$  
( \{ B^{{\bf H}}_n\}_{n\in \omega} ,0,\{ I_{mn} \}_{m<n} ,
\{ \lambda_c \}_{c\in\mathbb{Q}[i]}, +,-,\langle \rangle_{Re} , 
\langle \rangle_{Im}, U_1 ,U_2 ,\ldots , U_n  )   
$$ 
be the dynamical space constructed for $G$ in Theorem  \ref{hyplin}.
Then for any word $w(\bar{g})$ the statement 
$$
\mathsf{sup}_{v\in B_1} d(v, w(\overline{U})v) \le 0  
$$ 
is satisfied in this space if and only if $G\models w(\bar{g})=1$. 
Notice that when $G\models w(\bar{g})=1$ 
the statement above follows from $T_G$. 
This gives the reduction of  the word problem to 
the set of $0$-statements of $T_G$.  
\end{proof}

\begin{remark} 
 It is worth noting that formulas used in the proof of Theorem \ref{Novikov} 
are universal. 
This suggests considering the decidability problem 
for the universal theory of all dynamical (finite dimensional) Hilbert spaces. 
We will study this in Section 4.3. 
Note that Theorem \ref{Novikov} concerns a proper extension 
of it.   
\end{remark}

\begin{thrm} \label{Novikov2}
Assume that there is an algorithm 
which decides for every formula $\phi$ 
of the signature 
$$  
( \{ B_n\}_{n\in \omega} ,0,\{ I_{mn} \}_{m<n} ,
\{ \lambda_c \}_{c\in\mathbb{Q}[i]}, +,-,\langle \rangle_{Re} , 
\langle \rangle_{Im}, U_1 ,U_2 ,\ldots ,U_n  )  
$$ 
which has atomic subformulas only for $B_1$-variables 
and does not have free variables, 
whether $\phi^{o} = 2$ with respect to the theory 
of finite dimensional dynamical Hilbert spaces. 

Then there is not an MF finitely 
presented group $G= \langle g_1 ,...,g_n |\mathcal{ R}\rangle$ 
with undecidable word problem 
so that the amplification bound 
of its MF-approximations $\alpha_0$ is the function 
having only values $0$ and $2$ so that 
$\alpha_0 (g) = 0 \Leftrightarrow g=1$. 
\end{thrm} 

\begin{remark} 
It is worth mentioning that the maximal value of a formula 
$\phi$ of the theory of dynamical Hilbert spaces  
which has atomic subformulas only for $B_1$-variables 
and which does not have free variables, 
is $2$ (see Section 1).    
\end{remark}

\begin{proof} ({\em Theorem \ref{Novikov2}}) 
Assume the contrary. 
Let $G= \langle g_1 ,...,g_n |\mathcal{ R}\rangle$ be 
an MF finitely presented group as in the formulation 
and let $\alpha_0$ be the corresponding amplification bound. 
For any word $w(\bar{g})$ 
consider the formula 
$$ 
\phi_w = \mathsf{sup}_{v\in B_1} d(v, w(\overline{U})v) \dot{-} 
\mathsf{max} \{ \mathsf{sup}_{v\in B_1} d(v, w'(\overline{U})v)  
| \mbox{  where } 
w'(\bar{g})\in \mathcal{ R} \}. 
$$
If $w(\bar{g})$ is not equal to $1$ in $G$ then 
$\alpha_0 (w(\bar{g})) = 2$. 
By the argument of the final paragraph of the proof 
of Theorem \ref{hyplin} the set of all statements 
$2 \dot{-} \phi_w \le \varepsilon$ is finitely satisfiable
with respect to the theory of finite dimensional dynamical Hilbert spaces. 
By compactness (see Section 2 of \cite{BYU}) we see that 
$(\phi_w )^o =2$ with respect to this theory.

On the other hand note that in the case 
$G\models w(\bar{g})=1$ the equality $(\phi_w )^o =2$ 
does not hold. 
Indeed if a dynamical Hilbert space as in the formulation 
realizes this equality then all the statements    
$$
\mathsf{sup}_{v\in B_1} d(v, w' (\overline{U})v) \le 0 \mbox{ , where } 
w' (\bar{g})\in \mathcal{ R} ,  
$$ 
are also realized in it 
(by $\mathsf{sup}_{v\in B_1} d(v, w(\overline{U})v) \le 2$). 
Then by the definition of $G$ the map $w(\overline{U})$ 
defines the identity operator in this structure, a contradiction.  

We now see that the problem of the equality  
$(\phi_w )^o = 2$ with respect to the theory 
of finite dimensional dynamical Hilbert spaces is not decidable. 
This is a contradiction with our assumption. 
\end{proof}

\subsection{Universal theory} 

In this section we study the universal theory of 
finite dimensional dynamical Hilbert spaces 
with unitary operators $U_1 ,\ldots U_t$. 
The following proposition is the crucial observation of this section. 
It is obviously related to Theorem \ref{hyplin}.  

\begin{prop} \label{emb} 
Any dynamical Hilbert space is embeddable into 
a metric ultraproduct of finite dimensional Hilbert spaces. 
\end{prop} 

\begin{proof} 
Let $M$ be a dynamical Hilbert space. 
It  suffices to show that 
for every rational $\varepsilon$, every quantifier-free formula $\psi (\bar{x})$ 
and every tuple $\bar{c}\in M$ 
there is a finite  dimensional $N$ and $\bar{c}' \in N$ such that 
$|\psi (\bar{c})^M -\psi (\bar{c}' )^N | \le \varepsilon$.  
Indeed having this we can approximate all $\psi (\bar{c})^M$ 
by values of some $\psi (\bar{c}')$ in finite dimensional spaces 
and then just apply the  version of {\L}o\'{s}'s theorem for metric ultraproducts. 

Applying the L\"{o}wenheim-Skolem theorem if necessary 
we can arrange that $\bar{c}$ is taken from a separable $M$. 
We may assume that all terms appearing in $\psi (\bar{c})$ 
belong to a finite dimensional 
subspace $L < M$.  
Applying arguments of Section 7 of \cite{rosendal} 
we find a countable algebraically closed subfield 
$\mathcal{Q} <\mathbb{C}$ which is closed under complex conjugation, 
and a countable dense $\mathcal{Q}$-subspace $M' <M$ 
containing $L$ such that the inner product and the norm 
on $M'$ takes values in $\mathcal{Q}$.  
Since the formula $\psi (\bar{z})$ is a uniformly continuous function 
on $M$ we may approximate the operators $U_1, \ldots , U_t$ 
on $L$ by unitary operators say $U'_1, \ldots ,U'_t$ on $M'$ 
so that $\psi (\bar{c})^M$ is sufficiently close 
to $\psi (\bar{c})$ in $M'$. 
We now apply Lemma 7.4 of \cite{rosendal} and 
make $U'_1 ,\ldots ,U'_t$ finitary. 
As a result we obtain a finite dimensional 
dynamical $\mathcal{Q}$-subspace 
$M'' < M'$ so that the value of $\psi (\bar{c})$ in $M''$ belongs to    
$[\psi (\bar{c})^M -\varepsilon ,  \psi (\bar{c})^M +\varepsilon  ]$.  
It is densely contained in a finite dimensional dynamical 
Hilbert space over $\mathbb{C}$, say $N$.  
This finishes the proof. 
\end{proof}

Preserving the notation of Section 3.1 let $T_{f.d}$ be the theory of 
all finite dimensional dynamical Hilbert spaces 
with unitary operators $U_1 ,\ldots U_t$.
Let us consider the universal  (i.e. $\mathsf{sup}$)-sentences of this 
theory. 
The following corollary of Proposition \ref{emb} states that their values  
coincide with ones corresponding to 
the theory of all dynamical Hilbert spaces.   

\begin{prop} \label{rose} 
Let $\phi$ be a universal sentence 
of the language of dynamical Hilbert spaces.  
Then $\phi^{\circ}$ with respect to 
the theory of all dynamical Hilbert spaces 
coincides with $\phi^{\circ}$ with respect to $T_{f.d}$. 
\end{prop}  

\begin{proof} 
We may assume that all possible values of 
$\phi$ in dynamical Hilbert spaces belong to $[0,2]$.  
Let $r$ be the value of $\phi^{\circ}$ with respect to 
the theory of all dynamical Hilbert spaces and 
$r'$ be the value of $\phi^{\circ}$ with respect to 
the theory $T_{f.d}$. 
It is clear that $r' \le r$. 
To see $r\le r'$ assume the contrary and 
find a rational number $q \in [r',r]$. 
Then there is a separable dynamical Hilbert space $M$ 
such that $q < \phi^{M}$.  
This is equivalent to the condition that 
in $M$ the existential formula 
$2 \dot{-} \phi$ has the value which is less than $2 - q$. 
Since $M$ is embeddable into a metric ultraproduct of 
finite dimensional dynamical Hilbert spaces 
(by Proposition \ref{emb}), 
it is clear that $2\dot{-}\phi$ has value $< 2-q$ 
in some finite dimensional space, i.e. $\phi$ takes value $>r'$ in this space, 
a contradiction.  
\end{proof}

Using decidability of the theory of algebraically closed fields 
we fix an effective indexation of all $t$-tuples of 
unitary matrices of algebraic numbers so that any tuple 
consists of matrices of the same dimension. 
If $\bar{C}$ is such a tuple let $H_{\bar{C}}$ 
be the corresponding dynamical Hilbert space 
of the same dimension as the dimension of matrices in $\bar{C}$, 
say $n$. 
The following statements describe $\bar{C}$ (corresponding to $\bar{U}$):  
$$ 
\mathsf{inf}_{x_1 ,...,x_n\in B_1} \mathsf{max} (\mathsf{max}_{1\le i<j\le n} (|\langle x_i ,x_j\rangle -\delta_{i,j} |)   
$$
$$ 
(\mbox{ where } \delta_{i,j} \in \{ 0,1\} \mbox{ with } \delta_{i,j} =1 \leftrightarrow i=j ) \mbox{ , }
$$  
$$ 
\mathsf{max}_{1\le l\le n} \mathsf{max}_{1\le j\le t} (\parallel U_j (x_l )-\sum \lambda_{c_{j,l,k}} (x_k ) \parallel \dot{-} \varepsilon_l ))\le 0 , 
$$ 
$$
\mbox{ where } \varepsilon_l \in \mathbb{Q} 
\mbox{ and } c_{j,l,k} \in \mathbb{Q}[i] 
\mbox{ are appropriate approximations } 
$$ 
$$ 
\mbox{ of entries of matrices for } 
C_1 ,\ldots ,C_t . 
$$ 
For every $\bar{C}$ we fix a computable sequence 
of such axioms, say $\Sigma_{\bar{C}}$. 
Let $\hat{T}$ be the extension of the theory of all 
dynamical Hilbert spaces obtained by the additional 
family axioms consisting of the union 
of all $\Sigma_{\bar{C}}$. 
We see that $\hat{T}$ is computably axiomatizable.  
Let 
$$ 
\hat{H} = \bigoplus \{ H_{\bar{C}} : \bar{C} 
\mbox{ is a tuple of unitary matrices of algebraic numbers } \} .
$$  
Then $\hat{H} \models \hat{T}$. 
Since any finite dimensional dynamical Hilbert space 
is embeddable into a metric ultraproduct of spaces of the form $H_{\bar{C}}$, 
it is also embeddable into an ultrapower of $\hat{H}$. 
On the other hand $\hat{H}$ is embeddable into a ultraproduct of all $H_{\bar{C}}$. 
Thus applying Proposition \ref{rose} we have the first statement 
of the following lemma. 

\begin{lem} \label{value} 
(a) For any universal sentence $\phi$ the value $\phi^{\circ}$ 
with respect to $\hat{T}$ coincides  
with the value $\phi^{\circ}$ with respect to the theory of 
all finite dimensional dynamical Hilbert spaces. 
The latter value coincides with the value  $\phi^{\circ}$ 
with respect to the theory $Th(\hat{H})$. 

(b) For every existential sentence $\psi$ all values of $\psi$ in all models 
of $\hat{T}$ are the same. 
\end{lem} 

\begin{proof} 
(b) Since any dynamical Hilbert space is embeddable into 
a ultrapower of $\hat{H}$ the value of $\psi$ in $\hat{H}$ is minimal among 
all possible values. 
On the other hand since any model $M\models \hat{T}$ contains all $H_{\bar{C}}$ 
the space $\hat{H}$ is embeddable into a ultrapower of $M$. 
This shows that the values of $\psi$ in $M$ and $\hat{H}$ are the same. 
\end{proof} 

We can now prove the main result of Section 3.3. 

\begin{thrm} 
The universal theory of all dynamical Hilbert spaces is decidable. 
\end{thrm}

\begin{proof} 
According Lemma \ref{value}  if $\phi$ is a universal sentence, then for the theory $\hat{T}$ 
the value of the existential sentence $2 \dot{-} \phi$  coincides with $2 \dot{-} \phi^{\circ}$. 
Moreover the value $\phi^{\circ}$ is the same for $\hat{T}$ and for the theory of all dynamical Hilbert spaces. 
Thus the following algorithm always gives the result. 
Given universal $\phi$ and a rational $\varepsilon$, run all proofs from $\hat{T}$ and wait until you see that 
$\hat{T}\vdash \phi \dot{-} r$ and  $\hat{T}\vdash (2 \dot{-} \phi ) \dot{-} s$, where $r$ 
and $s$ are rational numbers so that  $r - (2-s)\le \varepsilon$. 
As a result one of the numbers $r$ or $s$ 
belongs to the interval of length $\varepsilon$ which contains 
$\phi^{\circ}$ with respect to the theory of all dynamical Hilbert spaces. 
\end{proof}

\subsection{Pseudocompactness} 

The author does not know if any dynamical Hilbert space 
is elementarily equivalent to an ultraproduct of 
finite dimensional dynamical Hilbert spaces 
(i.e. if its unit ball is pseudocompact). 
In this subsection we discuss some natural cases 
where we prove/expect the positive answer. 
Note that by the results of Section 4.1 
the global positive solution of this problem 
implies decidability of  the theory of 
all dynamical Hilbert spaces 
(and the theory of pseudo finite dimensional ones too). 
Moreover it also implies that any finitely generated group is MF 
(see the comment after Theorem \ref{hyplin}). 

According to \cite{ber} for any countable group $G$ 
the theory of all unitary $G$-representations 
has a model completion. 
It can be described as follows. 

Let $\infty H_{\rho}$ be a dynamical Hilbert space 
corresponding to a representation of $G$ which is maximal with 
respect to almost containedness $\prec$ (see Definition \ref{alco}). 
Theorem 2.11 of \cite{ber} states that 
the class of existentially closed unitary representations of $G$ 
is axiomatized by all $0$-$\mathsf{inf}$-statements which 
hold in $\infty H_{\rho}$. 
Moreover by the amalgamation property the theory 
of these representations is complete. 

Note all dynamical Hilbert spaces 
with $t$ operators form the class of all representations 
of the free $t$-generated group $F_t$. 
Thus the class of all existentially closed 
representations of $F_t$ coincides with the model 
completion of the theory of all dynamical Hilbert spaces.  
The following observation was pointed out to the author by 
E. Hrushovski 
(who applied a different argument). 

\begin{prop} 
Existentially closed dynamical Hilbert spaces are 
pseudo finite dimensional. 
\end{prop} 

\begin{proof} 
By Proposition \ref{emb} any existentially closed 
dynamical Hilbert space is contained in 
a pseudo finite dimensional one. 
All $0$-$\mathsf{inf}$-statements which 
hold in $\infty H_{\rho}$ also hold 
in the corresponding pseudo finite dimensional 
dynamical Hilbert space. 
By Theorem 2.11 of \cite{ber} it is also existentially closed 
and by completeness of the theory is elementarily 
equivalent to $\infty H_{\rho}$. 
\end{proof} 

This proposition justifies the following question. 

\begin{quote} {\em 
\item  Given finitely generated group $G$ are existentially 
closed dynamical Hilbert spaces corresponding to representations of $G$ 
pseudo finite dimensional?  }
\end{quote} 

By Theorem 2.15 of \cite{ber} a countable group $G$ is amenable 
if and only if the direct sum of countably many copies of $l^2 (G)$, 
the regular unitary representation of $G$, is existentially closed. 
Thus it is easy to see that in the amenable case 
the positive answer to this question follows from 
the statement that for any finitely generated group $G$ 
the dynamical space $l^2 (G)$ 
is pseudo finite dimensional.  
As we already know the latter statement implies 
that any finitely generated group is MF. 

We have the following partial result. 
This can be considered as a stronger version 
of the statement that any finitely generated LEF group 
is MF, see \cite{CDE}. 

\begin{prop} \label{LEF} 
Let $G$ be a finitely generated LEF group. 
Then the dynamical $G$-space $l^2 (G)$ is pseudo finite dimensional. 
\end{prop} 

Before the proof of the proposition we remind the reader that 
a group $H$ is called LEF \cite{GV} if for every finite 
subset $F\subseteq H$ there is a finite group $S$ containing $F$ 
so that for any $x,y,z \in F$ the equality 
$x\cdot y = z$ holds in $H$ if and only if it holds in $S$. 
Residually finite groups are LEF. 
This proposition together with Theorem 2.15 of \cite{ber} imply that 
when $G$ is an amenable LEF group generated by $t$ elements, 
all  dynamical Hilbert spaces which are existentially closed
$G$-representations are pseudo finite dimensional. 

\bigskip 

\begin{proof} 
Let $G$ be $t$-generated. 
We fix a tuple $g_1 ,\dots ,g_t$ of generators and consider  
$\Gamma_G$, the corresponding coloured Cayley graph of $G$.  
By the condition LEF for every natural $k$ 
there is a finite $t$-generated group $G_k$ 
such that the $k$-ball of $1$ in the Cayley graph of 
$G_k$ coincides with the $k$-ball of $1$ in $\Gamma_G$. 
It is also worth mentioning that for any two elements of 
$\Gamma_{G_k}$ their $k$-balls are naturally isomorphic. 

Let $\hat{H}_G$ be a non-principal metric ultraproduct 
of all dynamical spaces $l^2 (G_k )$. 
Then identifying any word $w(\bar{g})$ from $G$ 
with the sequence of the corrresponding elements 
from all $G_k$ we consider $l^2 (G)$ as a substructure of 
$\hat{H}_G$.  
Moreover each $g_i$ naturally defines a unitary transformation of 
$\hat{H}_G$. 
So the regular representation of $G$ naturally 
extends to a representation on $\hat{H}_G$. 
The ${\bf C}^*$-algebra generated by $G$ in the 
algebra of all bounded operators of $l^2 (G)$ is the operator norm    
closure of the $*$-algebra generated by the regular action 
of $G$, i.e. by $\mathbb{C}G$. 
We denote it by $C^* (G)$. 

It is easy to see that the $*$-algebra $\mathbb{C}G$ 
naturally acts on $\hat{H}_G$. 
Let us note that $C^* (G)$ also has a natural action on 
$\hat{H}_G$. 
Indeed let $p_1 , \ldots , p_i, \ldots$ be a 
sequence from $\mathbb{C} G$ which is  norm convergent. 
For every $i$ let $l_i$ be the maximal length of words 
from $\Gamma_G$ which appear in $p_1 ,\ldots ,p_i$. 
Then for every natural number $m$ there is a natural 
number $n$ such that for all $k>n$ all the $(l_m + 2m)$-balls 
of $1$ in $\Gamma_{G_k}$ are naturally isomorphic. 
In particular for any element $v \in \Gamma_{G_k}$ 
the partial actions of $p_1 ,\ldots , p_m$ 
inside the subspace of $l^2 (G_k )$ supported by   
the $(l_m + 2m)$-ball of $v$ correspond to the actions defined by 
$p_1 ,\ldots , p_m$ in the subspace of $l^2 (G)$ 
supported by the $(l_m + 2m)$-ball of $1$. 
This obviously implies that for any $i$ and $j$ 
$$
\lim_{k\rightarrow \infty} \parallel p_i - p_j \parallel_{l^2 (G_k )}  
= \parallel p_i - p_j \parallel_{l^2 (G)}. 
$$ 
In particular the sequence $p_1 , \ldots , p_i, \ldots$ 
is norm convergent in $\hat{H}_G$. 
As a result we see that $\hat{H}_G$ is a $C^* (G)$-module. 
 
We now apply Theorem 2.20 of \cite{argo}. 
It states that two representations of a ${\bf C}^*$-algebra $\mathcal{A}$ 
are elementarily equivalent in continuous logic if and only if 
for any $a\in \mathcal{A}$ the ranks of the corresponding 
elements are finite and the same or are infinite. 
It is easy to see that the arguments above can be applied for 
a verification that the $C^* (G)$-representations $l^2 (G)$ and $\hat{H}_G$ 
satisfy the conditions of this theorem. 
\end{proof} 
 
We finish this section by the observation  
that in the case of a single operator pseudocompactness 
follows from the spectral decomposition theorem.  

\begin{prop} 
Any dynamical Hilbert space with a single unitary operator is 
pseudo finite dimensional. 
\end{prop} 

This observation was suggested to the author by E. Hrushovski. 
Before the proof we remind the reader the 
{\em spectral decomposition theorem}. 
\begin{quote} 
Let $U$ be a unitary operator. 
Then there is a unique resolution of the identity 
$\{ E_r | r\in [0,2\pi ]\}$ such that 
$U^k = \int^{2\pi}_0 e^{i k t} dE_t$ for all $k\in \mathbb{Z}$. 
\end{quote}  
In this formulation each $E_r$ is a projection operator, 
$E_r = 0$ for $r\le 0$, $E_r = Id$ for $r> 2\pi$, and each 
$E_s - E_r$ is positive for $r<s$. 
The theorem should be interpreted as follows. 
Given $\varepsilon >0$ and $\delta > 0$ 
there is $n_0$ such that for any partition  
$0=r_0 <s_0 = r_1 < s_1 =r_2 < \ldots < s_{n_0} = 2\pi +\varepsilon$ 
with 
$\mathsf{max} \{ s_l - r_l \} < \frac{2 (2\pi + \varepsilon )}{n_0}$ 
we have 
$\parallel U - \sum^{n_0}_{l=1} e^{i k r_l} (E_{s_l} - E_{r_l}) \parallel <\delta$.

\begin{proof} 
Apply the spectral decomposition theorem to 
a dynamical Hilbert space $(H, U)$. 
All operators $E_r$ are self-ajoint, their images are closed subspaces 
and $E_r (H) \subset E_s (H)$ for $r<s$. 
In particular $E_s - E_r$ is an orthogonal projection operator. 
Given $\varepsilon >0$, $\delta >0$ and $n_0 (\varepsilon, \delta )$ 
as above for each natural number $n$ one can define a subspace of $H'<H$ 
of dimension $\le n (n_0 +1)$, where 
$$
\mathsf{dim} E_{r_1} (H') = \mathsf{min} (n, \mathsf{dim} E_{r_1} (H)),   
\ldots ,  
$$ 
$$
\mathsf{dim} (E_{s_l}- E_{r_l})(H') = 
\mathsf{min} (n, \mathsf{dim} (E_{s_l} - E_{r_l}) (H)), \ldots . 
$$ 
Then the formula $\sum^{n_0}_{l=1} e^{i k r_l} (E_{s_l} - E_{r_l})$ 
defines an operator on $H'$. 
We denote it by $U_{\varepsilon, \delta, n}$. 

Take a sequence $(\varepsilon_i , \delta_i , n_i )$ 
such that $\varepsilon_i \rightarrow 0$, $\delta_i \rightarrow 0$
and $n_i \rightarrow \infty$. 
Let $(\hat{H} , \hat{U} )$ be a non-principal metric ultraproduct 
of the corresponding structures $(H', U_{\varepsilon, \delta, n})$. 
In order to prove that $(\hat{H}, \hat{U})$ is elementarily equivalent to 
$(H,U)$ we apply the Henson's theorem (Theorem 3.1 of \cite{argober}) cited in 
Section 2.3 above. 
We also apply some part of the proof of Theorem 3.3 from \cite{argober}. 

Let $\mu \in S^1 \setminus \sigma (U)$ and $\eta = d(\mu, \sigma (U))$. 
Then the statement 
$$
\mathsf{sup}_{u} (\eta \parallel u\parallel \dot{-} \parallel U(u) - \mu u\parallel ) \le 0
$$ 
holds in $(H,U)$ and so holds its $2\delta$-approximation in  
$(H', U_{\varepsilon ,\delta ,n})$ as above. 
Now it is easy to see that the exact statement holds in 
$(\hat{H}, \hat{U})$. 
Therefore $\sigma (\hat{U} ) \subseteq \sigma (U)$. 
The same argument proves $\sigma (U ) \subseteq \sigma (\hat{U})$.

For each $\lambda \in \sigma (U)$ let us consider statements of the form  
$$
\mathsf{inf}_{u_1} \ldots \mathsf{inf}_{u_m} \mathsf{max}_{i,j} (|\langle u_i ,u_j\rangle |, |\parallel u_i \parallel - 1|, |U(u_i ) - \lambda u_i |) \le 0 . 
$$ 
Repeating the argument above we see that they 
do not distinguish $(H,U)$ and $(\hat{H}, \hat{U})$. 
When $\lambda$ is an isolated point in $\sigma (U)$ 
this implies that the dimensions of 
$\{ x\in H: U (x) = \lambda x \}$ 
and $\{ x\in \hat{H}: \hat{U} (x) = \lambda x \}$ 
are the same. 
Thus the conditions of the Henson's theorem are satisfied. 
\end{proof}

\section{Dynamical $n$-qubit spaces}

In this section we demonstrate how the method 
of interpretability works in some expansions of dynamical spaces. 
In paragraphs (A) - (C) below we describe why these expansions are natural 
from the point of view of quantum computations. 

\subsection{Preliminaries} 

\paragraph{(A)} 
We remind the reader that states of quantum systems 
are represented by normed vectors of tensor products 
$$ 
(...(\mathcal{B}_1 \bigotimes \mathcal{B}_2 )\bigotimes .... )\bigotimes \mathcal{B}_k ,
$$ 
$$
\mbox{  where } \mathcal{B}_i \cong \mathbb{C}\bigoplus \mathbb{C} 
\mbox{ under isomorphisms of Hilbert spaces, } i\le k. 
$$ 
In Dirac's notation elements of $\mathcal{B}_i$ 
are denoted by $|h\rangle$ and 
tensors 
$$ 
(...(|h_1 \rangle \otimes |h_2 \rangle ) ...) \otimes |h_k \rangle 
\mbox{ are denoted by } 
|h_1 h_2 ...h_k \rangle . 
$$      
Any normed $h\in \mathcal{B}_i$ is called a {\em qubit};  
it is a linear combination of  
$|0\rangle = (1,0)$ and $|1\rangle =(0,1)$.  

The probability amplitude $a (\phi \rightarrow \psi)$ 
is defined as the inner product 
$\langle \psi |\phi\rangle$   
and the probability $p(\phi \rightarrow \psi)$ 
is $|a(\phi \rightarrow \psi)|^2$. 
Dynamical evolutions of the quantum system are 
represented by unitary operators on  
$\mathcal{B}^{\otimes k}$.

\bigskip 

\paragraph{(B)} 
It is worth noting that  continuous logic can be considered as 
a theory in some extension ( {\bf RPL}$\forall$ ) 
of {\L}ukasiewicz logic (see \cite{DGP}). 
The latter is traditionally linked with quantum mechanics, 
\cite{CCGP}, \cite{pykacz}. 
Thus the idea that continuous logic should enter 
into the field is quite natural.

We will consider dynamical $n$-qubit spaces in 
continuous logic as follows. 
Firstly we extend structures of complex Hilbert spaces by 
additional discrete sort $Q$ with $\{ 0,1\}$-metric  
and a map $qu: Q\rightarrow B_1$ so that 
the set $qu(Q)$ is an orthonormal basis of $\mathbb{H}$. 

When $Q$ consists of $2^n$ elements we may denote them 
by $|i_0 ...i_{n-1} \rangle$ with $i_j \in \{ 0,1\}$. 
In Quantum Computations this set is called 
the {\em computational basis of the system} and  
$\mathcal{B}^{\otimes n}$ is called the $n$-{\em qubit space}. 
Secondly we enrich the structure $(qu , Q, \mathcal{B}^{\otimes n})$ 
by unitary operators $U_1 ,...,U_t$.  
We call it a {\em dynamical $n$-qubit space}. 
It turns out that the condition $|Q| = 2^n$ 
is not essential. 
For example one can consider subspaces generated by arbitrary subsets 
of the computational  basis. 
Therefore we will consider {\em marked Hilbert spaces} and {\em marked dynamical Hilbert spaces}, 
i.e. (dynamical) Hilbert spaces expanded by a discrete sort $Q$ 
and a map $qu$ which injectively maps $Q$ onto an orthonormal basis of the space.  
In this section we study the following problem.  
\begin{quote}  
{\em Describe classes of marked dynamical Hilbert spaces 
having decidable continuous theory. }
\end{quote} 
Below we give examples of classes of 
marked dynamical spaces 
with undecidable sets of $0$-statements. 
This material is based on the method of interpretability 
described in the second part of Section 3. 
Comparing these results with Section 4  
it is worth noting that in fact 
(following the approach of Quantum Computations) 
we extend the language used in Section 4 by 
a discrete unary predicate $Q$. 
We will see that this procedure is essential 
for the expressive power of the language:   
there are natural subclasses of marked dynamical 
Hilbert spaces where undecidable first order theories of 
some classes of finite structures 
can be interpreted on $Q$.

\bigskip

\paragraph{(C)} 
It is worth noting that a dynamical $n$-qubit space 
defines a family of quantum automata over the language 
$\{ 1,...,t\}^{*}$, where each automaton 
is determined by the $2^n$-dimensional  
diagonal matrix $P$ of the projection to final states.   
Fixing  $\lambda\in \mathbb{Q}$ we say that 
a word $w=i_1 ...i_k$ is {\em accepted} 
by the corresponding $P$-automaton if 
$$
ACC_w = \parallel P U_{i_k} ... U_{i_1} |0^{\otimes n}\rangle \parallel^2 >\lambda . 
$$   
These issues are described in \cite{gruska}, \cite{MC} and \cite{DJK}. 
The corresponding algorithmic problems 
were in particular studied in the paper of 
H. Derksen, E. Jeandel, P. Koiran  \cite{DJK}.  
They have proved that the following problems are decidable for 
$U_1 ,...,U_t$  over finite extensions of $\mathbb{Q}[i]$: \\ 
(i) Is there $w$ such that $ACC_w >\lambda$? \\ 
(ii) Is a threshold $\lambda$ isolated, i.e. is there $\varepsilon$ 
that for all $w$,  $|ACC_w - \lambda |\ge \varepsilon$ ? \\ 
(iii) Is there a threshold $\lambda$ which is isolated? 

The observation that  given $P$ each statement 
$ACC_w \le \lambda$  or $|ACC_w - \lambda |\ge \varepsilon$
can be rewritten as a continuous statement 
of the theory of dynamical $n$-qubit spaces 
partially motivated our research in this paper. 

\subsection{Interpretations} 

We start this section with an undecidability 
result of some classes of constant expansions of marked  Hilbert  spaces. 
Then we apply the idea of the proof 
to a more interesting example of a class of 
marked dynamical spaces. 
In Remark \ref{naturality} we comment how these classes are natural. 

\begin{thrm}  \label{constants} 
There is a class of marked Hilbert spaces expanded by four constants, 
i.e. structures of the form
$$  
(Q, qu , \{ B_n\}_{n\in \omega} ,0,\{ I_{mn} \}_{m<n} ,
\{ \lambda_c \}_{c\in\mathbb{Q}[i]}, +,-,\langle \rangle_{Re} , \langle \rangle_{Im}, a_1 ,a_2 , b_1 ,b_2  )  , 
$$
which is distinguished in the class of all marked Hilbert spaces 
with $\langle b_1 ,b_2 \rangle \not= 0$ by a single continuous 
statement and which has undecidable set of all $0$-statements. 
\end{thrm} 

\begin{proof} 
Consider the following formula: 
$$ 
\psi (x ,y_1 ,y_2 ) = |\langle qu(y_1 ),x\rangle - \langle qu(y_2 ),x \rangle | ,  
$$ 
where $y_1 ,y_2$ are variables of the sort $Q$ and $x$ is of the sort $B_1$. 

In any marked Hilbert  space any equivalence relation on $Q$ can be realised by 
$\psi (a , y_1 ,y_2 )\le 0$ for appropriately chosen $a\in B_1$:   
define $a$ to be a linear combination of $qu (q_l )$, so that for 
equivalent $q_j$ and $q_k$ the coefficients of $qu (q_j )$ and $qu (q_k )$ 
in $a$ are the same. 

Let us introduce the following formula: 
$$ 
\psi^c (x , z_1 ,z_2 ,y_1 ,y_2 ) = 
| \langle z_1 ,z_2 \rangle |\dot{-} |\langle qu(y_1 ),x\rangle - \langle qu(y_2 ),x \rangle | , 
$$ 
where $y_1 ,y_2$ are variables of the sort $Q$ and $x,z_1 ,z_2$ are of $B_1$.

If $a$ defines an equivalence relation on $Q$ as above 
then there are $b_1 ,b_2 \in B_1$ with sufficiently small 
$|\langle b_1 ,b_2 \rangle |\not=0$ 
(in fact here we only need 
$$
|\langle b_1 ,b_2 \rangle | < \mathsf{min} (|\langle qu(y_1 ),a\rangle - \langle qu(y_2 ),a \rangle | : y_1 \mbox{ and } y_2 \mbox{ are not equivalent } \mbox { ) }  
$$
and satisfying 
$$ 
\mathsf{sup}_{y_1 ,y_2} \mathsf{min}(\psi (a, y_1 ,y_2 ), \psi^c (a,b_1 ,b_2 , y_1 , y_2 ))\le 0
$$ 
$$
\mathsf{sup}_{y_1 ,y_2} (| \langle b_1 ,b_2 \rangle |\dot{-} (\psi (a, y_1 ,y_2 )+ \psi^c (a,b_1 ,b_2 , y_1 , y_2 )))\le 0.
$$
We see that the formula 
$\psi^c (a,b_1 ,b_2, y_1 , y_2 )$ can be 
interpreted as the complement of the equivalence relation 
defined by $\psi (a, y_1 , y_2 )$ in the class of these marked spaces. 

This allows us to define interpretability of the first-order theory of 
finite structures of two equivalence relations  
(which is undecidable by Proposition 5.1.7 from \cite{ershov}) 
in the class, say $\mathcal{K}$, of marked Hilbert spaces extended by constants 
$a_1 ,a_2 ,b_1 ,b_2$ where $b_1 ,b_2$ satisfy the statements 
above for both $a_1$ and $a_2$ instead of $a$. 

In fact we axiomatize $\mathcal{K}$ in the class of all marked Hilbert spaces 
with $\langle b_1 , b_2 \rangle \not= 0$ by the following continuous statements: 
$$ 
\mathsf{sup}_{y_1 ,y_2} \mathsf{min}(\psi (a_i , y_1 ,y_2 ), 
| \langle b_1 ,b_2 \rangle |\dot{-} |\langle qu(y_1 ),a_i \rangle - \langle qu(y_2 ),a_i \rangle |) \le 0 
\mbox{ , where } i=1,2.  
$$ 

In terms of  Theorem \ref{unde} the formulas $\phi^+ , \phi^- , \theta^+ , \theta^-$ 
become degenerate: $\phi^-$ can be taken as $d(y,y)$ for 
the (descrete) sort $Q$, then 
$$
\phi^+ (y) = 1 \dot{-} d(y,y) \mbox{ , } 
\theta^- (y_1 ,y_2 ) = d(y_1 , y_2 ) \mbox{ , } 
\theta^+ (y_1 ,y_2 ) = 1 - d(y_1 ,y_2 ). 
$$ 
Formulas   $\psi (a_1 , y_1 , y_2 )$ and $\psi^c (a_1 ,b_1 ,b_2, y_1 , y_2 )$ 
play the role of $\psi^{-}_1$  and $\psi^+_1$. 
Then $\psi (a_2 , y_1 , y_2 )$ and $\psi^c (a_2 ,b_1 ,b_2, y_1 , y_2 )$
play the role of  $\psi^{-}_2$  and $\psi^+_2$. 

To each formula $\rho (\bar{y})$ of the theory of 
two equivalence relations so that the quantifier-free part 
is in the disjunctive normal form we associate 
the appropriately rewritten continuous formula 
$\rho^* (a_1 ,a_2 ,b_1 ,b_2 , \bar{y})$ 
(where we use $\mathsf{min}$ and $\mathsf{max}$ 
instead of $\vee$ and $\wedge$, 
and we exchange $+$ and $-$ when $\neg$ appears). 
Since the free variables of the latter, $\bar{y}$, 
are of the sort $Q$, when $\rho$ is quantifier-free  
the values $\rho^* (a_1 ,a_2 ,b_1 ,b_2 , \bar{c})$ belong to 
$\{ 0\} \cup [ | \langle b_1 ,b_2 \rangle |, 1]$.  
It is easy to see that in structures of $\mathcal{K}$ 
the same property holds for any 
formula $\rho (\bar{y})$. 

This obviously implies that when 
$\rho$ is a sentence,  the sentence 
$$ 
\mathsf{min} (| \langle b_1 ,b_2 \rangle |,\rho^*(a_1 ,a_2 ,b_1 ,b_2 ))
$$   
has the following property:  
\begin{quote}  
$\rho$ is satisfied in all finite models of two equivalence relations 
if and only if all structures of $\mathcal{K}$  satisfy 
$\mathsf{min} (| \langle b_1 ,b_2 \rangle |,\rho^*(\bar{a},\bar{b})) =0$.  
\end{quote} 
By Theorem \ref{unde} this gives a required reduction. 
\end{proof} 

\begin{remark} \label{naturality} 
The class of marked Hilbert spaces considered in Theorem \ref{constants} 
is the intersection of an axiomatizable class of continuous structures 
with the complement of an axiomatizable class (defined by 
the inequality $\langle b_1 , b_2 \rangle \not=0$). 
Both classes of this intersection are natural and easily defined.  
It is also clear that the axiomatizable closure of this class  
still has undecidable set of $0$-statements. 
On the other hand we do not have any reasonable description of 
the members of this closure. 
The following theorem concerns a very similar situation 
(in a different language).  
The author would very like to find an easily described 
axiomatizable class of marked dynamical Hilbert spaces 
with undecidable theory 
(or at least with undecidable set of $0$-statements). 
\end{remark}

\begin{thrm} 
There is a class of marked dynamical Hilbert spaces  
in the signature 
$$  
(Q, qu , \{ B_n\}_{n\in \omega} ,0,\{ I_{mn} \}_{m<n} ,
\{ \lambda_c \}_{c\in\mathbb{Q}[i]}, +,-,\langle \rangle_{Re} , 
\langle \rangle_{Im}, U_1 ,U_2 , U_3 ,U_4 ,U_5  ) ,   
$$
which is distinguished in the class of all marked dynamical Hilbert spaces with 
$$
\mathsf{sup}_v d( U_3 (v) ,v )\not= 0
$$ 
by a single continuous statement and the set of $0$-statements 
of the continuous theory of which is not computable. 
\end{thrm} 

\begin{proof} 
We will use the construction of Theorem \ref{constants} 
with some necessary changes. 
For example we replace the value  
$| \langle b_1 ,b_2 \rangle |$ from  
that theorem  by  $\mathsf{sup}_{v} d(U_3 (v ),v )$. 
The  constants $a_1$ and $a_2$  will appear as 
the normed vectors fixed by $U_1$ and $U_2$ 
respectively. 
Although we choose $U_i$, $i=1,2$, so that 
the subspace of fixed vectors of $U_i$ 
coincides with $\mathbb{C} a_i$, 
we cannot define these constants by a continuous formula. 
This is why some additional values will be 
used in the proof.  
The values  $\mathsf{sup}_{v} d(U_3 (v ),v )$  
and $\mathsf{sup}_{v} d(U_4 (v ),v )$ will appear 
in 'fuzzy' versions of formulas from 
the proof of Theorem \ref{constants}. 

Let: 
$$ 
\psi_i (y_1 ,y_2 ) = \mathsf{sup}_{u} \mathsf{min} ( 
\mathsf{sup}_{v_1}(d(U_{3} (v_1 ),v_1 )) \dot{-} \mathsf{max} (d(U_i (u ),u ),|1-\parallel u \parallel |),
$$ 
$$    
(|\langle qu(y_1 ),u\rangle - \langle qu(y_2 ),u \rangle | \dot{-} 
\mathsf{sup}_{v_2}d(U_4 (v_2 ),v_2 ))),  
$$ 
where $y_1 ,y_2$ are variables of the sort $Q$, $i\in \{ 1,2\}$  
and $u , v_1 ,v_2$ are of the sort $B_1$. 

To see that {\em in any marked dynamical Hilbert space any equivalence 
relation on $Q$ can be realized by} $\psi_1 (y_1 ,y_2 )\le 0$ 
let us define $a_1$ to be a linear combination of 
$qu (q_l )$ of length 1, so that for equivalent 
$q_j$ and $q_k$ the coefficients of $qu (q_j )$ 
and $qu (q_k )$ in $a_1$ are the same 
(the case of $\psi_2$ and $a_2$ is similar). 
We also fix a rational number $r$ (for both $a_1$ and $a_2$)  
so that for non-equivalent $q_j$ and $q_k$ the coefficients 
of $qu (q_j )$ and $qu (q_k )$ in $a_1$ 
are distant by $> r$.   
Note that any $e^{i\phi} a_1$ has the same properties 
as $a_1$ with respect to elements of $qu(Q)$. 

We extend $a_1$ to an orthonormal basis of 
the space and define $U_1$ to be a unitary 
operator having the vectors of the basis as eigenvectors 
so that $\mathbb{C} a_1$ is the subspace of fixed points. 
The remaining eigenvalues are chosen in the form $e^{i \varphi}$ 
so that the corresponding eigenvectors are taken by $U_1$ 
at the distance  $\ge 1/10$ (i.e.$|1 - e^{i \varphi}|\ge 1/10$). 

We will assume that $\mathsf{sup}_{v\in B_1}d(U_{3} (v),v )>0$ 
in our structures. 
Choosing $U_4$ we demand that 
$\mathsf{sup}_{v\in B_1}(d(U_4 (v ),v ))$ 
is much less than $r$. 
It follows  that when $q_i$ and $q_j$ are not equivalent  
the value $u=a_1$ realizes 
the inequality $\psi_1 (q_i , q_j )>0$. 

Having $U_4$ we construct $U_3$  
so close to $Id$ (with respect to the operator norm) 
that the statement $\psi_1 (y_1 , y_2 )\le 0$ indeed realizes 
the equivalence relation we consider.  
For this we only need the condition that 
if $q_i$ and $q_j$ are equivalent and a vector $c$ satisfies   
$$    
|\langle qu(q_i ),c\rangle - \langle qu(q_j ),c \rangle | >  
\mathsf{sup}_{v\in B_1}d(U_4 (v ),v ))
\mbox{ and } |1-\parallel c \parallel |< \mathsf{sup}_{v\in B_1}(d(U_{3} (v ),v ))    
$$ 
i.e. the projection of  $c$ to $qu(q_i )-qu(q_j )$ 
and the length of $c$ are  sufficiently large
(i.e. $c$ is sufficiently distant from  $a_1$), 
then  $\mathsf{sup}_{v\in B_1}(d(U_{3} (v ),v )) \le d(U_1 (c), c) $.

Let us now introduce $U_5$ with 
$r =  \mathsf{sup}_{v}(d(U_{5} (v ),v)$ 
and consider the following formulas for $i = 1,2$: 
$$ 
\psi^c_i (y_1 ,y_2 ) =  \mathsf{sup}_{u} \mathsf{min} ( 
\mathsf{sup}_{v_1}(d(U_{3} (v_1 ),v_1 )) \dot{-} \mathsf{max} (d(U_i (u ),u ),|1-\parallel u \parallel |),
$$ 
$$ 
(\mathsf{sup}_{v_2}d (U_5 (v_2 ),v_2 )\dot{-} |\langle qu(y_1 ),u\rangle - \langle qu(y_2 ),u \rangle | ) ), 
$$ 
where $y_1 ,y_2$ are variables of the sort $Q$ and $u,v_1 ,v_2$ are of $B_1$.
If necessary we may correct $U_3$ making $\mathsf{sup}_{v}d(U_{3} (v),v)$ smaller  
so that the following statement holds. 
$$ 
\mathsf{sup}_{y_1 ,y_2} \mathsf{min}(\psi_i (y_1 ,y_2 ), \psi^c_i ( y_1 , y_2 ))\le 0 ,
$$ 
Verifying this one can apply the argument of the previous paragraph. 
Similar reasoning implies that 
$$
\mathsf{sup}_{y_1 ,y_2} ( \mathsf{sup}_{v}d(U_{3} (v),v )\dot{-} (\psi_i ( y_1 ,y_2 )+ \psi^c_i ( y_1 , y_2 )))\le 0.
$$
As before the formula $\psi^c_i (y_1 , y_2 )$ will be 
interpreted as the complement of the equivalence relation 
defined by $\psi_i (y_1 , y_2 )$ in the class of these qubit spaces. 

This allows us to define interpretability of the (undecidable) 
first-order theory of finite structures of two equivalence relations in 
the class, say $\mathcal{K}$, of marked dynamical spaces with respect to 
operators  $U_1 ,U_2 ,U_3 ,U_4 , U_5$. 
The formulas 
$\phi^+ , \phi^- , \theta^+ , \theta^-$ 
(see  Theorem \ref{unde} ) are taken as in 
Theorem \ref{constants} (i.e. $\phi^- (y) = d(y,y)$  
and $\theta^- (y_1 ,y_2 ) = d(y_1 , y_2 )$).  
Formulas   $\psi_i (y_1 , y_2 )$ and $\psi^c _i ( y_1 , y_2 )$ 
play the role of $\psi^{-}_i$  and $\psi^+_i$ for  $i=1, 2$. 

To each formula $\rho (\bar{y})$ of the theory of two equivalence relations 
so that the quantifier-free part is in the disjunctive normal form we associate 
the appropriately rewritten continuous formula 
$\rho^* ( \bar{y})$. 
Since the free variables of the latter $\bar{y}$ are of the sort $Q$,  
when $\rho$ is quantifier-free,  the values 
$\rho^* ( \bar{c})$ belong to 
$\{ 0\} \cup [ \mathsf{sup}_{v}d(U_{3} (v),v ), 1]$.  
Thus we see that in structures of $\mathcal{K}$ 
the same property holds for any 
formula $\rho (\bar{y})$. 

This obviously implies that when 
$\rho$ is a sentence,  the sentence 
$$ 
\mathsf{min} (\mathsf{sup}_{v}d(U_{3} (v),v ),\rho^*)
$$   
has the following property:  
\begin{quote}  
$\rho$ is satisfied in all finite models of two equivalence relations 
if and only if all structures of $\mathcal{K}$  satisfy 
$\mathsf{min} (\mathsf{sup}_{v}d(U_{3} (v),v ),\rho^*) =0$.  
\end{quote} 
This finishes the proof. 
\end{proof}

\end{document}